\documentclass[preprint, 12pt]{elsarticle}
\usepackage{indentfirst,enumerate,amssymb,amsfonts,amsmath,amsthm,mathrsfs,dsfont}
\usepackage{color}
\usepackage[colorlinks=true,linkcolor=blue,citecolor=blue]{hyperref}
\usepackage{enumerate}
\usepackage{geometry}

\numberwithin{equation}{section}

\theoremstyle{plain}
\newtheorem{theorem}{Theorem}[section]
\newtheorem{lemma}[theorem]{Lemma}
\newtheorem{proposition}[theorem]{Proposition}
\newtheorem{corollary}[theorem]{Corollary}
\newtheorem{definition}[theorem]{Definition}

\theoremstyle{definition}
\newtheorem{example}[theorem]{Example}
\newtheorem{remark}{Remark}

\allowdisplaybreaks

\journal{arXiv}

\begin{document}
	
\begin{frontmatter}
		
\title{Denjoy-Carleman solvability of Vekua-type periodic operators }

\author[ufpr]{Alexandre Kirilov}
\affiliation[ufpr]{organization={	Departamento de Matemática, Universidade Federal do Paraná},
			city={Curitiba},
			postcode={81531-990}, 
			state={Paraná},
			country={Brazil}}
\ead{akirilov@ufpr.br}

\author[ufpr]{Wagner Augusto Almeida de Moraes}
\ead{wagnermoraes@ufpr.br}
	
	\author[ppgm]{Pedro Meyer Tokoro}
\affiliation[ppgm]{organization={Programa de Pós-Graduação em Matemática, 
		Universidade Federal do Paraná},
	city={Curitiba},
	postcode={81531-990}, 
	state={Paraná},
	country={Brazil}}
\ead{pedro.tokoro@ufpr.br}
	
		\begin{abstract}
		This paper explores the solvability and global hypoellipticity of Vekua-type differential operators on the n-dimensional torus, within the framework of Denjoy-Carleman ultradifferentiability. We provide the necessary and sufficient conditions for achieving these global properties in the case of constant-coefficient operators, along with applications to classical operators. Additionally, we investigate a class of variable coefficients and establish conditions for its solvability.
	\end{abstract}


\end{frontmatter}

\section{Introduction}

	The theory of generalized analytic functions, as introduced by Vekua in \cite{vekua}, initially dealt with solutions to equations of the form:
	$$ \partial_{\bar{z}}u + Au + B\bar{u} = F, $$
	where $\partial_{\bar{z}}$ stands for the Cauchy-Riemann operator, and the coefficients $A$ and $B$ belong to a suitable function space in the complex plane. This theory is closely linked to the theory of holomorphic functions and relies on specific regularity conditions on the coefficients to apply the similarity principle. Vekua used this framework to address problems in the theory of shells and to explore problems of infinitesimal bendings of surfaces.
	
	In \cite{Kravchenko2009}, V. Kravchenko extended this theory by replacing the Cauchy-Riemann operator with more general complex-valued vector fields. This generalization enabled the study of classical equations from mathematical physics, including the Schrödinger, Dirac, and Maxwell equations, among others.
	
	The work of Vekua, Kravchenko, Bers, and other researchers has found extensive applications in various fields, such as boundary value problems in elasticity theory, hydrodynamics, electric potential, mechanics, and more.
	
	In the article \cite{KMT24indag}, we investigated the solvability and global hypoellipticity of the periodic Vekua-type operator $P$ defined by
	\begin{equation}\label{P}
		Pu = Lu - Au - B\bar{u},
	\end{equation}
	where $A$ and $B$ are complex constants, and $L:\mathcal{C}^\infty(\mathbb{T}^n) \to \mathcal{C}^\infty(\mathbb{T}^n)$ is a general constant coefficient partial differential operator. In this case, we obtained a complete characterization of these global properties for this class of operators.
	
	Now, in this article, we address the Denjoy-Carleman solvability problem for the same class of constant coefficient operators. Additionally, we extend the analysis to a class of operators with variable coefficients, building on the results obtained in \cite{BDM} and \cite{AD}.
	
	In this paper, we organize the content as follows. In Section 2, we revisit fundamental notations and properties related to Denjoy-Carleman classes of ultradifferentiable functions. We provide the necessary background and refer the reader to \cite{artigo_bruno} for more comprehensive details and proofs. Next, we delve into the analysis of operators with constant coefficients, outlining the conditions under which they are solvable in the ultradifferentiable sense. This involves establishing a necessary and sufficient condition, linked to a Diophantine condition, which relates the growth of a value associated with the symbol of the operator to its solvability. We also demonstrate that global hypoellipticity and solvability are equivalent properties.
	
	In the latter part of the paper, we shift our focus to operators with variable coefficients. Motivated by the work in \cite{BDM} and \cite{AD}, we consider a class of complex vector fields with variable coefficients that satisfy the Nirenberg-Treves condition $(P)$. This condition ensures the local solvability of this class of vector fields. We derive conditions under which these operators are solvable.

	Further details on smooth and ultradifferentiable solvability and the regularity of solutions on tori can be found in \cite{AKM19, BDG17jfaa,BDG18, LA20}. In the wider context of compact Lie groups, the reader can see some results in \cite{wagner,KMR21komatsu, KMR22komatsu} and the references cited therein.

	\section{Notations and Fourier analysis in Denjoy-Carleman classes}

	Let us begin by recalling the definition and some useful properties of the spaces of ultradifferentiable functions of Roumieu type defined on the $n$-dimensional torus $\mathbb{T}^n \simeq \mathbb{R}^n / 2\pi \mathbb{Z}^n$. For more details on the results and their proofs, we refer the reader to \cite{artigo_bruno}.
	
	A sequence $\mathscr{M} = \{m_j\}_{j \in \mathbb{N}}$ of positive real numbers is called a weight sequence if it satisfies the following conditions:
	\begin{enumerate}\label{Msequence}
		\item[$i.$] $m_0 = m_1 = 1$;
		\item[$ii.$] $m_j^2 \leq m_{j-1} \,m_{j+1}$ for all $j \in \mathbb{N}$;
		\item[$iii.$] $\displaystyle\sup_{j,k \in \mathbb{N}_0} \left(\dfrac{m_{j+k}}{m_j \,m_k}\right)^{\frac{1}{j+k}} \leq H$, for some $H \geq 1$.
	\end{enumerate}
	
	Given a weight sequence $\mathscr{M} = \{m_j\}_{j \in \mathbb{N}}$ and $h > 0$, consider the set $\mathcal{E}_{\mathscr{M},h}(\mathbb{T}^n)$ of all smooth functions $f: \mathbb{T}^n \to \mathbb{C}$ such that
	$$
	\|f\|_{\mathscr{M},h} \doteq \sup_{\alpha \in \mathbb{N}_0^n} \sup_{x \in \mathbb{T}^n} \, \dfrac{|\partial^\alpha f(x)|}{h^{|\alpha|} \,m_{|\alpha|} \,|\alpha|!} < \infty.
	$$
	
	We observe that $\mathcal{E}_{\mathscr{M},h}(\mathbb{T}^n)$ is a Banach space with respect to the norm $\|\cdot\|_{\mathscr{M},h}$, and we define the space of periodic $\mathscr{M}$-ultradifferentiable functions of Roumieu type by
	$$
	\mathcal{E}_{\mathscr{M}}(\mathbb{T}^n) = \underset{h \rightarrow +\infty}{\operatorname{ind} \lim} \, \mathcal{E}_{\mathscr{M}, h}(\mathbb{T}^n).
	$$
	
	Moreover, its dual space, denoted by $\mathscr{D}'_\mathscr{M}(\mathbb{T}^n)$, is commonly referred to as the space of periodic $\mathscr{M}$-ultradistributions.

	Therefore, a smooth function $f: \mathbb{T}^n \to \mathbb{C}$ belongs to the space $\mathcal{E}_\mathscr{M}(\mathbb{T}^n)$ if and only if there exist constants $C, h > 0$ such that, for all $\alpha \in \mathbb{N}_0^n$,
	$$
	\sup_{x \in \mathbb{T}^n} |\partial^\alpha f(x)| \leq C h^{|\alpha|} m_{|\alpha|} |\alpha|!.
	$$
	
	The best-known example of a space of ultradifferentiable functions is given by the weight sequence $\mathscr{M} = \{m_j\}_{j \in \mathbb{N}_0}$, where $m_j = (j!)^{s-1}$. For each fixed $s \geq 1$, $\mathcal{E}_\mathscr{M}(\mathbb{T}^n) = \mathcal{G}^s(\mathbb{T}^n)$ is the space of periodic Gevrey functions of order $s$.
	
	The Fourier coefficients of a function $f \in \mathcal{E}_\mathscr{M}(\mathbb{T}^n)$ are given by
	$$
	\widehat{f}(\xi) \doteq (2\pi)^{-n} \int_{\mathbb{T}^n} f(x) e^{-i \xi \cdot x} \, dx, \quad \xi \in \mathbb{Z}^n,
	$$
	and, as usual, the Fourier coefficients of an ultradistribution $u \in \mathscr{D}'_\mathscr{M}(\mathbb{T}^n)$ are defined by
	$$
	\widehat{u}(\xi) = (2\pi)^{-n} \langle u, e^{-i \xi \cdot x} \rangle, \quad \xi \in \mathbb{Z}^n.
	$$
	
	The classes of ultradifferentiable functions and ultradistributions on tori can be characterized by the rate of decay or growth of their Fourier coefficients in the following way:
	\begin{align}
		f \in \mathcal{E}_\mathscr{M}(\mathbb{T}^n) \Longleftrightarrow & \ \text{there exist } C, \delta > 0 \text{ such that,} \label{EM_fourier_total} \\
		& \qquad \qquad |\widehat{f}(\xi)| \leq C \inf_{j \in \mathbb{N}_0} \dfrac{m_j \,j!}{\delta^j \,(1 + \|\xi\|)^j}, \quad \xi \in \mathbb{Z}^n; \nonumber \\
		u \in \mathscr{D}'_\mathscr{M}(\mathbb{T}^n) \Longleftrightarrow & \ \text{for all } \varepsilon > 0, \text{ there is } C_\varepsilon > 0 \text{ such that,} \label{DM_fourier_total} \\
		& \qquad\, \qquad |\widehat{u}(\xi)| \leq C_\varepsilon \sup_{j \in \mathbb{N}_0}  \dfrac{\varepsilon^j (1 + \|\xi\|)^j}{m_j \,j!}, \quad \xi \in \mathbb{Z}^n. \nonumber
	\end{align}
	
	Moreover, when $f \in \mathcal{E}_\mathscr{M}(\mathbb{T}^n)$ and $u \in \mathscr{D}'_\mathscr{M}(\mathbb{T}^n)$, we have
	$$
	f(x) = \sum_{\xi \in \mathbb{Z}^n} \widehat{f}(\xi) e^{i \xi \cdot x} \quad \text{and} \quad u = \sum_{\xi \in \mathbb{Z}^n} \widehat{u}(\xi) e^{i \xi \cdot x},
	$$
	with convergences in $\mathcal{E}_\mathscr{M}(\mathbb{T}^n)$ and $\mathscr{D}'_\mathscr{M}(\mathbb{T}^n)$ respectively.

	Another important characterization for the development of this work is given by the partial Fourier series, which can be described as follows: let us write $n = p + q$, with $p, q \in \mathbb{N}$, consider the direct sum $\mathbb{R}^n = \mathbb{R}^p \oplus \mathbb{R}^q$ and write $(t, x) \in \mathbb{T}^n$, meaning that $t \in \mathbb{T}^p$ and $x \in \mathbb{T}^q$. 
	
	Given $f \in \mathcal{E}_\mathscr{M}(\mathbb{T}^n)$, for each $t \in \mathbb{T}^p$, consider the function $f_t: \mathbb{T}^q \to \mathbb{C}$ given by $f_t(x) = f(t, x)$. Clearly, for each $t \in \mathbb{T}^p$, we have $f_t \in \mathcal{E}_\mathscr{M}(\mathbb{T}^q)$ and by the characterization \eqref{EM_fourier_total}, we can write
	$$
	f(t, x) = f_t(x) = \sum_{\xi \in \mathbb{Z}^q} \widehat{f}(t, \xi) e^{i \xi \cdot x}, 
	$$
	with convergence in $\mathcal{E}_\mathscr{M}(\mathbb{T}^q)$, where
	$$
	\widehat{f}(t, \xi) \doteq (2\pi)^{-q} \int_{\mathbb{T}^q} f_t(x) e^{-i \xi \cdot x} \, dx.
	$$
	Moreover, for each $\xi \in \mathbb{Z}^q$ fixed, we have that $\widehat{f}(\cdot, \xi) \in \mathcal{E}_\mathscr{M}(\mathbb{T}^p)$.
	
	Analogously, given an ultradistribution $u \in \mathscr{D}'_\mathscr{M}(\mathbb{T}^n)$, for each $\xi \in \mathbb{Z}^q$, we define the partial Fourier coefficient of $u$ as the functional $\widehat{u}(t, \xi): \mathcal{E}_\mathscr{M}(\mathbb{T}^p) \to \mathbb{C}$ given by
	$$
	\langle \widehat{u}(t, \xi), f \rangle = (2\pi)^{-q} \langle u, f(t) \otimes e^{-i \xi \cdot x} \rangle, \quad f \in \mathcal{E}_\mathscr{M}(\mathbb{T}^p).
	$$
	
	In this framework we have the following characterizations:
	\begin{align}
		f \in \mathcal{E}_\mathscr{M}(\mathbb{T}^n) \Longleftrightarrow\ & \exists C_p,h_p,h_q > 0;\ \forall \alpha\in\mathbb{N}_0^p,\ \forall\xi\in\mathbb{Z}^q \label{EM_fourier_partial} \\
		& \ \sup_{t\in\mathbb{T}^p}|\partial^\alpha \widehat f(t,\xi)| \leq C_p\,h_p^{|\alpha|}\,m_{|\alpha|}\,|\alpha|!\,\inf_{j\in\mathbb{N}_0}\dfrac{m_j\,j!}{h_q^j\,(1+\|\xi\|)^j}; \nonumber \\
		u \in \mathscr{D}'_\mathscr{M}(\mathbb{T}^n) \Longleftrightarrow\ & \forall \varepsilon > 0,\ \exists C_\varepsilon > 0;\ \forall f\in\mathcal{E}_\mathscr{M}(\mathbb{T}^p),\ \forall \xi\in\mathbb{Z}^q \label{DM_fourier_total2} \\
		& \ \,  |\langle \widehat u(t,\xi),f\rangle |\leq C_{\varepsilon,h}\|f\|_{\mathscr{M},h}\,\sup_{j\in\mathbb{N}_0}\dfrac{\varepsilon^j\,(1+\|\xi\|)^j}{m_j\,j!}. \nonumber
	\end{align}

	\section{Constant coefficient Operators}\label{ccoef_ultra}
	
	Consider the operator $P: \mathcal{E}_{\mathscr{M}}(\mathbb{T}^n) \to \mathcal{E}_{\mathscr{M}}(\mathbb{T}^n)$ defined by
	\begin{equation}\label{P_ultra}
		Pu = Lu - Au - B\bar{u},
	\end{equation}
	where $A, B \in \mathbb{C}$, and $L$ is an operator of the form
	$$L = \sum_{0<|\alpha|\leqslant m} c_\alpha \partial^\alpha,$$
	with $c_\alpha \in \mathbb{C}$ for all $\alpha \in \mathbb{N}_0^n$ satisfying $0 < |\alpha| \leqslant m$, and with symbol 	
	$$\sigma_L(\xi) = \sum_{0<|\alpha|\leqslant m} i^{|\alpha|} c_\alpha \xi^\alpha, \ \xi \in \mathbb{Z}^n.$$
	
	By the continuity of $P$, if $u(x) = \sum_{\xi \in \mathbb{Z}^n} \widehat{u}(\xi)e^{i\xi\cdot x}\in \mathscr{D}'_{\mathscr{M}}(\mathbb{T}^n)$, we have
	\begin{align*}
		Pu(x) = & \sum_{\xi \in \mathbb{Z}^n} \left(\sigma_L(\xi) - A\right) \widehat{u}(\xi) e^{i\xi\cdot x} - B \sum_{\xi \in \mathbb{Z}^n} \overline{\widehat{u}(-\xi)} e^{-i\xi\cdot x}.
	\end{align*}
	
	Therefore, the Fourier coefficients of any solution of the equation $Pu = f$ must satisfy
	$$(\sigma_L(\xi) - A) \widehat{u}(\xi) - B \overline{\widehat{u}(-\xi)} = \widehat{f}(\xi).$$
	
	Taking the conjugate of the previous equation for $-\xi \in \mathbb{Z}^n$, we obtain
	$$
	\overline{\widehat{f}(-\xi)} = (\overline{\sigma_L(-\xi)} - \bar{A}) \overline{\widehat{u}(-\xi)} - \bar{B} \widehat{u}(\xi),
	$$
	which gives us, for each $\xi \in \mathbb{Z}^n$, the following linear system:
	$$
	\begin{cases}
		\left(\sigma_L(\xi) - A\right) \widehat{u}(\xi) - B \overline{\widehat{u}(-\xi)} = \widehat{f}(\xi), \\
		-\bar{B} \widehat{u}(\xi) + (\overline{\sigma_L(-\xi)} - \bar{A}) \overline{\widehat{u}(-\xi)} = \overline{\widehat{f}(-\xi)}.
	\end{cases}
	$$
	
	Solving this system for $\widehat{u}(\xi)$, we obtain
	\begin{equation}\label{DeltaJuJ_ultra} 
		\Delta_\xi \widehat{u}(\xi) = (\overline{\sigma_L(-\xi)} - \bar{A}) \widehat{f}(\xi) + B \overline{\widehat{f}(-\xi)},
	\end{equation}
	where $\Delta_\xi$ is the discriminant of the above system, that is,
	\begin{equation}\label{DeltaJ_ultra} 
		\Delta_\xi = \left(\sigma_L(\xi) - A\right)(\overline{\sigma_L(-\xi)} - \bar{A}) - |B|^2.
	\end{equation}
	
	Observe that $\overline{\Delta_\xi} = \Delta_{-\xi}$ for all $\xi \in \mathbb{Z}^n$. In particular, $\Delta_\xi = 0$ if, and only if, $\Delta_{-\xi} = 0$.
	
	Inspired by the references \cite{BDM,AD,KMT24indag}, we introduce the following notion of solvability and the main result of this section.
	
	\begin{definition}
		We say that the differential operator $P: \mathcal{E}_{\mathscr{M}}(\mathbb{T}^n) \to \mathcal{E}_{\mathscr{M}}(\mathbb{T}^n)$ defined in (\ref{P_ultra}) is $\mathscr{M}$-solvable if there exists a subspace $\mathcal{F} \subset \mathcal{E}_{\mathscr{M}}(\mathbb{T}^n)$ of finite codimension such that, for all $f \in \mathcal{F}$, there exists $u \in \mathcal{E}_{\mathscr{M}}(\mathbb{T}^n)$ satisfying $Pu = f$.
	\end{definition}
	
	\begin{theorem}\label{teo_ccoef_gen_ultra}
		The operator $P$ is $\mathscr{M}$-solvable if and only if for all $\varepsilon > 0$, there exist constants $C_\varepsilon, \gamma_\varepsilon > 0$ such that
		\begin{equation}
			\|\xi\| \geq \gamma_\varepsilon \ \Rightarrow \ |\Delta_\xi| \geq C_\varepsilon \inf_{j \in \mathbb{N}_0}  \dfrac{m_j \,j!}{\varepsilon^j (1 + \|\xi\|)^j} . \tag{DC$\mathscr{M}$} \label{DC_M-condition}
		\end{equation}
	\end{theorem}
	

	\begin{proof}
		Firstly, observe that 
		\begin{align*}
			\sup_{j \in \mathbb{N}_0} \dfrac{\varepsilon^j (1 + \|\xi\|)^j}{m_j \,j!} \leq \sum_{j \in \mathbb{N}_0} \dfrac{\varepsilon^j (1 + \|\xi\|)^j}{m_j \,j!} \leq \sum_{j \in \mathbb{N}_0} \dfrac{\varepsilon^j (1 + \|\xi\|)^j}{j!} = e^{\varepsilon (1 + \|\xi\|)}.
		\end{align*}
		
		Therefore, for any $\varepsilon > 0$ we have 
		\begin{align}
			\inf_{j \in \mathbb{N}_0} \dfrac{m_j \,j!}{\varepsilon^j (1 + \|\xi\|)^j} = \left( \sup_{j \in \mathbb{N}_0} \dfrac{\varepsilon^j (1 + \|\xi\|)^j}{m_j \,j!} \right)^{-1} > 0 \label{inf-positive}
		\end{align}
		
		Let us assume that \eqref{DC_M-condition} holds. Then, the set
		\begin{equation*} 
			\Omega = \{\xi \in \mathbb{Z}^n \, : \, \Delta_\xi = 0\}
		\end{equation*}
		is finite. This implies that the subspace
		\begin{equation*} 
			\mathcal{F} = \{f \in \mathcal{E}_{\mathscr{M}}(\mathbb{T}^n) \, : \, \widehat{f}(\xi) = 0 \text{ for all } \xi \in \Omega\}
		\end{equation*}
		has finite codimension in $\mathcal{E}_{\mathscr{M}}(\mathbb{T}^n)$.
		
		Given $f \in \mathcal{F}$, by \eqref{EM_fourier_total}, there exist $C, \delta > 0$ such that, for all $\xi \in \mathbb{Z}^n$,
		\begin{equation} 
			|\widehat{f}(\xi)| \leq C \,\inf_{j \in \mathbb{N}_0} \dfrac{m_j \,j!}{\delta^j (1 + \|\xi\|)^j}. \label{f-hat-the-solv}
		\end{equation}
		
		Let $H > 0$ be the value specified in condition {\it iii.} of the definition of a weight sequence on page \pageref{Msequence}, and set $\varepsilon = \delta / H$ in condition \eqref{DC_M-condition}. Then, there exist constants $C_\varepsilon, \gamma_\varepsilon > 0$ such that
		\begin{equation} 
			\|\xi\| \geq \gamma_\varepsilon \Rightarrow |\Delta_\xi| \geq C_\varepsilon \,\inf_{j \in \mathbb{N}_0} \dfrac{m_j \,j!}{\varepsilon^j (1 + \|\xi\|)^j}. \label{Delta-thm-solv}
		\end{equation}

	Let $u \in \mathscr{D}_{\mathscr{M}}'(\mathbb{T}^n)$ be a solution of $Pu = f$. For $\|\xi\| \geq \gamma_\varepsilon$, we have $\xi \in \mathbb{Z}^n \setminus \Omega$. Therefore, it follows from \eqref{DeltaJuJ_ultra}, \eqref{f-hat-the-solv}, and \eqref{Delta-thm-solv} that
	\begin{align*}
		|\widehat{u}(\xi)| &= |\Delta_\xi|^{-1} \big|[\overline{\sigma_L(-\xi)} - \bar{A}]\widehat{f}(\xi) + B\overline{\widehat{f}(-\xi)}\big| \\
		&\leq \dfrac{C}{C_\varepsilon} \left(|\overline{\sigma_L(-\xi)} - \bar{A}| + |B|\right) \inf_{j \in \mathbb{N}_0} \dfrac{m_j \,j!}{\delta^j (1 + \|\xi\|)^j} \sup_{j \in \mathbb{N}_0} \dfrac{(\delta/H)^j (1 + \|\xi\|)^j}{m_j \,j!}.
	\end{align*}
	
	Notice that $\sigma_L$ is a polynomial of degree $N$, where $N$ is the order of the operator $L$. Thus, there exists a constant $C > 0$ such that
	\begin{equation*}
		|\overline{\sigma_L(-\xi)} - \bar{A}| + |B| \leq C(1 + \|\xi\|)^N, \quad \xi \in \mathbb{Z}^n.
	\end{equation*}
	
	Moreover, it follows from Lemma \ref{lema2.25} that
	\begin{align*}
		& \inf_{j \in \mathbb{N}_0} \dfrac{m_j \,j!}{\delta^j (1 + \|\xi\|)^j} \ \sup_{j \in \mathbb{N}_0} \dfrac{(\delta/H)^j (1 + \|\xi\|)^j}{m_j \,j!} \\ 
		&\leq \left(\inf_{j \in \mathbb{N}_0} \dfrac{m_j \,j!}{(\delta/H)^j (1 + \|\xi\|)^j}\right)^2 \left(\inf_{j \in \mathbb{N}_0} \dfrac{m_j \,j!}{(\delta/H)^j (1 + \|\xi\|)^j}\right)^{-1} \\
		&= \inf_{j \in \mathbb{N}_0} \dfrac{m_j \,j!}{(\delta/H)^j (1 + \|\xi\|)^j}.
	\end{align*}

	Again, from Lemma \ref{lema2.25}, we have
	\begin{align*}
		|\widehat u(\xi)|\ \ \leq\ \ & C(1+\|\xi\|)^N\inf_{j\in\mathbb{N}_0}\dfrac{m_j\,j!}{(\delta/H)^j(1+\|\xi\|)^j}\\
		\leq\ \ & C(1+\|\xi\|)^N \left(\inf_{j\in\mathbb{N}_0}\dfrac{m_j\,j!}{(\delta/H^2)^j (1+\|\xi\|)^j}\right)^2\\
		\leq\ \ & C(1+\|\xi\|)^N \ \dfrac{m_N\,N!}{(\delta/H^2)^N(1+\|\xi\|)^N} \ \inf_{j\in\mathbb{N}_0}\dfrac{m_j\,j!}{(\delta/H^2)^j(1+\|\xi\|)^j}\\
		\leq\ \ & C \, h^N\, m_N\, N!\, \inf_{j\in\mathbb{N}_0}\dfrac{m_j\,j!}{\delta_0^j(1+\|\xi\|)^j},
	\end{align*}
	with $h=H/\delta$ and $\delta_0=\delta/H^2$. 
	
	Therefore, $u\in\mathcal{E}_{\mathscr{M}}(\mathbb{T}^n)$ and $Pu=f$, that is, $P$ is $\mathscr{M}$-solvable.

On the other hand, suppose that \eqref{DC_M-condition} does not hold. Then, there is $\varepsilon > 0$ such that for each $\ell \in \mathbb{N}$, there is $\xi_\ell \in \mathbb{Z}^n$ with $\|\xi_\ell\| \geq \ell$ and 
\begin{equation}\label{DeltaJ_decay_ultra}
	|\Delta_{\xi_\ell}| < \inf_{j \in \mathbb{N}_0} \dfrac{m_j \,j!}{\varepsilon^j (1 + \|\xi_\ell\|)^j}.
\end{equation}

Denoting $\xi_{m\ell}$ as the $m$-th coordinate of $\xi_\ell$, we can choose $\{\xi_\ell\}_{\ell \in \mathbb{N}}$ such that all $\xi_{m\ell}$ are non-zero and have the same sign for some $m \in \{1, \dots, n\}$.

	\medskip\noindent
	\textbf{Case 1:} $\Delta_{\xi_\ell}=0$ for infinitely many $\ell\in\mathbb{N}$.
	
	Passing to a subsequence, we may assume that $\Delta_{\xi_\ell}=0$ for all $\ell\in\mathbb{N}$. Thus, from \eqref{DeltaJuJ_ultra}, we have
		$(\overline{\sigma_L(-\xi)}-\bar A) \widehat{f}(\xi) + B\overline{\widehat{f}(-\xi)} = 0,$ for $\xi\in\Omega.$
	
	If $B \neq 0$, from \eqref{DeltaJ_ultra}, we have $(\sigma_L(\xi_\ell)-A)\,(\overline{\sigma_L(-\xi_\ell)}-\bar A) = |B|^2$ for all $\ell\in\mathbb{N}$. Since $|B|^2 \neq 0$, we conclude that $\sigma_L(\xi_\ell)-A \neq 0$ and $\overline{\sigma_L(-\xi_\ell)} - \bar A \neq 0$ for all $\ell\in\mathbb{N}$. In this context, there are infinitely many compatibility conditions for the Fourier coefficients of $f \in \mathcal{C}^\infty(\mathbb{T}^n)$ to satisfy $Pu = f$. Consequently, $P$ is not $\mathscr{M}$-solvable.
	
	If $B = 0$, the Fourier coefficients of $u$ and $f$ must satisfy $(\sigma_L(\xi) - A) \widehat{u}(\xi) = \widehat{f}(\xi)$ and $(\overline{\sigma_L(-\xi)} - \bar A) \overline{\widehat{u}(-\xi)} = \overline{\widehat{f}(-\xi)}$ for all $\xi \in \mathbb{Z}^n$. Since $[\sigma_L(\xi_\ell) - A] \,[\overline{\sigma_L(-\xi_\ell)} - \bar A] = \Delta_{\xi_\ell} = 0$ for all $\ell\in\mathbb{N}$, at least one of the following must hold:  
	$\sigma_L(\xi_\ell) - A = 0 \quad \text{or} \quad \overline{\sigma_L(-\xi_\ell)} - \bar A = 0,$ 
	for each $\ell \in \mathbb{N}$. This implies that for infinitely many indices $\xi \in \mathbb{Z}^n$, either $\widehat{f}(\xi) = 0$ or $\widehat{f}(-\xi) = 0$. Therefore, there are infinitely many compatibility conditions for the Fourier coefficients of $f$ such that $Pu = f$ has a smooth solution. Consequently, $P$ is not solvable.

\medskip\noindent
\textbf{Case 2:} $\Delta_{\xi_\ell} = 0$ for a finite number of indices $\ell \in \mathbb{N}$.

Passing to a subsequence, we may assume that $\Delta_{\xi_\ell} \neq 0$ for all $\ell \in \mathbb{N}$, and consider the set $\Omega = \{\xi_\ell \in \mathbb{Z}^n : \ell \in \mathbb{N}\}$. Observe that $\xi \in \Omega \Leftrightarrow -\xi \notin \Omega$, and take an infinite subset $\Omega_0 \subset \Omega$.

Let us assume first that $B \neq 0$ and consider the function 
$$
f(x) = \sum_{\xi \in \Omega_0} \Delta_\xi e^{i \xi \cdot x}.
$$

It follows from \eqref{DeltaJ_decay_ultra} and \eqref{EM_fourier_total} that $f \in \mathcal{E}_{\mathscr{M}}(\mathbb{T}^n)$. If $u \in \mathscr{D}_{\mathscr{M}}'(\mathbb{T}^n)$ is a solution of $Pu = f$, proceeding as in \cite{KMT24indag}, the projection of $u$ on the subspace $\mathscr{D}_{\mathscr{M}}'(\mathbb{T}^n)$ generated by the frequencies $\pm \Omega_0$ is given by
$$
v(x) = \sum_{\xi \in \Omega_0} B e^{-i \xi \cdot x} + \sum_{\xi \in \Omega_0} [\overline{\sigma_L(-\xi)} - \bar A] e^{i \xi \cdot x} \in \mathscr{D}'(\mathbb{T}^n) \setminus C^\infty(\mathbb{T}^n).
$$
Then, $v \in \mathscr{D}_{\mathscr{M}}'(\mathbb{T}^n) \setminus \mathcal{E}_{\mathscr{M}}(\mathbb{T}^n)$, which implies $u \in \mathscr{D}_{\mathscr{M}}'(\mathbb{T}^n) \setminus \mathcal{E}_{\mathscr{M}}(\mathbb{T}^n)$.

Now consider the case where $B = 0$. For all $\xi \in \Omega_0$, observe that
$$
\left( \inf_{j \in \mathbb{N}_0} \dfrac{m_j \,j!}{(\varepsilon / H)^j (1 + \|\xi\|)^j} \right)^2 \geq \inf_{j \in \mathbb{N}_0} \dfrac{m_j \,j!}{\varepsilon^j (1 + \|\xi\|)^j} > |\Delta_\xi| = |\sigma_L(\xi) - A| \,|\sigma_L(-\xi) - A|.
$$

Therefore, either
$$
|\sigma_L(\xi) - A| < \inf_{j \in \mathbb{N}_0} \dfrac{m_j \,j!}{(\varepsilon / H)^j (1 + \|\xi\|)^j}  \quad \text{or} \quad 
|\sigma_L(-\xi) - A| < \inf_{j \in \mathbb{N}_0} \dfrac{m_j \,j!}{(\varepsilon / H)^j (1 + \|\xi\|)^j}
$$
holds for infinitely many $\xi \in \Omega_0$.

If 
$$
|\sigma_L(\xi) - A| < \inf_{j \in \mathbb{N}_0} \dfrac{m_j \,j!}{(\varepsilon / H)^j (1 + \|\xi\|)^j}
$$
holds for infinitely many $\xi \in \Omega_0$, we can pass to a subsequence and assume that this inequality holds for all $\xi \in \Omega_0$. This implies that the function 
$$
f(x) = \sum_{\xi \in \Omega_0} [\sigma_L(\xi) - A] e^{i \xi \cdot x}
$$
belongs to $\mathcal{E}_{\mathscr{M}}(\mathbb{T}^n)$. If $u \in \mathscr{D}_{\mathscr{M}}'(\mathbb{T}^n)$ is a solution of $Pu = f$, proceeding as in \cite{KMT24indag}, the projection of $u$ on the subspace $\mathscr{D}_{\mathscr{M}}'(\mathbb{T}^n)$ generated by the frequencies in $\Omega_0$ is given by
$$
v(x) = \sum_{\xi \in \Omega_0} e^{i \xi \cdot x} \in \mathscr{D}'(\mathbb{T}^n) \setminus C^{\infty}(\mathbb{T}^n).
$$
Thus, $v \in \mathscr{D}_{\mathscr{M}}'(\mathbb{T}^n) \setminus \mathcal{E}_{\mathscr{M}}(\mathbb{T}^n)$, which implies $u \notin \mathcal{E}_{\mathscr{M}}(\mathbb{T}^n)$.

The case where
$$
|\sigma_L(-\xi) - A| < \inf_{j \in \mathbb{N}_0} \dfrac{m_j \,j!}{(\varepsilon / H)^j (1 + \|\xi\|)^j}
$$
holds for infinitely many $\xi \in \Omega_0$ is analogous.

Finally, note that this construction holds for any infinite subset $\Omega_0 \subset \Omega$, providing us with infinitely many linearly independent functions $f \in \mathcal{E}_{\mathscr{M}}(\mathbb{T}^n)$ such that there are no $u \in \mathcal{E}_{\mathscr{M}}(\mathbb{T}^n)$ satisfying $Pu = f$. Therefore, $P$ is not $\mathscr{M}$-solvable.
\end{proof}

\begin{definition}
	We say that $P$ is $\mathscr{M}$-globally hypoelliptic when the conditions
	$$u\in\mathscr{D}'_\mathscr{M}(\mathbb{T}^n) \mbox{ and }  Pu\in\mathcal{E}_{\mathscr{M}} (\mathbb{T}^n) \mbox{ imply }  u\in\mathcal{E}_{\mathscr{M}}(\mathbb{T}^n).$$
\end{definition}

\begin{corollary}\label{coro_GH_ultra}
	$P$ is $\mathscr{M}$-solvable if, and only if, $P$ is $\mathscr{M}$-globally hypoelliptic.
\end{corollary}
\begin{proof}
If $P$ is $\mathscr{M}$-solvable, we have shown in Theorem \ref{teo_ccoef_gen_ultra} that for any $f\in\mathcal{E}_{\mathscr{M}}(\mathbb{T}^n)$, if $u\in\mathscr{D}_{\mathscr{M}}'(\mathbb{T}^{n})$ is a solution of $Pu=f$, then $u\in\mathcal{E}_{\mathscr{M}}(\mathbb{T}^n)$. Therefore, $P$ is $\mathscr{M}$-globally hypoelliptic.

Conversely, if $P$ is not $\mathscr{M}$-solvable, according to Theorem \ref{teo_ccoef_gen_ultra}, we can find a function $f\in\mathcal{E}_{\mathscr{M}}(\mathbb{T}^n)$ such that if $u\in \mathscr{D}_{\mathscr{M}}'(\mathbb{T}^{n})$ is a solution of $Pu=f$, then $u\notin \mathcal{E}_{\mathscr{M}}(\mathbb{T}^{n})$. Therefore, $P$ is not $\mathscr{M}$-globally hypoelliptic.
\end{proof}

In the papers \cite{BDM} and \cite{KMT24indag}, the notion of (smooth) solvability was introduced as follows:
An operator $P: \mathcal{C}^\infty(\mathbb{T}^n) \to \mathcal{C}^\infty(\mathbb{T}^n)$ is said to be solvable if there exists a subspace $\mathcal{F} \subset \mathcal{C}^\infty(\mathbb{T}^n)$ of finite codimension such that for all $f \in \mathcal{F}$, there exists $u \in \mathcal{C}^\infty(\mathbb{T}^n)$ such that $Pu = f$. 

Moreover, it was shown that $P$ is solvable if and only if there exists $\gamma > 0$ such that 
$$
\|\xi\| \geq \gamma \ \Rightarrow \ |\Delta_\xi| \geq (1 + \|\xi\|)^{-\gamma}.
$$

Now, observe that for any $\varepsilon > 0$ we have
$$
\dfrac{1}{(1 + \|\xi\|)^{\gamma}} = \dfrac{\varepsilon^\gamma}{m_\gamma \,\gamma!} \,\dfrac{m_\gamma \,\gamma!}{\varepsilon^\gamma (1 + \|\xi\|)^{\gamma}} \geq C_\varepsilon \inf_{j \in \mathbb{N}_0} \dfrac{m_j \,j!}{\varepsilon^j (1 + \|\xi\|)^j},
$$
where $C_\varepsilon = \dfrac{\varepsilon^\gamma}{m_\gamma \,\gamma!}$. Therefore, \eqref{DC_M-condition} holds, and we have proved the following result:

\begin{proposition}\label{smooth_ultra}
	If $P$ is solvable, then $P$ is $\mathscr{M}$-solvable.
\end{proposition}

However, the reciprocal of this statement is generally not true. In \cite{AD}, the authors provided a counterexample involving a constant-coefficient vector field and continued fractions in the context of Gevrey spaces.

\begin{corollary}\label{teo_elliptic_ultra}
	If $L$ is an elliptic differential operator, then $P$ is $\mathscr{M}$-solvable and $\mathscr{M}$-globally hypoelliptic.
\end{corollary}

\begin{proof}
	It follows from \cite[Theorem 4]{KMT24indag} that $P$ is solvable. Therefore, by Corollary \ref{coro_GH_ultra} and Proposition \ref{smooth_ultra}, $P$ is $\mathscr{M}$-solvable and $\mathscr{M}$-globally hypoelliptic.
\end{proof}

\subsection{Examples}

\begin{example}[Laplace operator]
	If $ L = \sum_{j=1}^{n} {\partial^2}/{\partial x_j^2} $, then the operator $ P $ is $\mathscr{M}$-solvable and $\mathscr{M}$-globally hypoelliptic because $ L $ is elliptic.
\end{example}

\begin{example}[Heat operator]
	If $ L = {\partial}/{\partial t} - \eta^2 \sum_{j=1}^{n} {\partial^2}/{\partial x_j^2} $, where $\eta > 0$, then $ P $ is $\mathscr{M}$-solvable and $\mathscr{M}$-globally hypoelliptic because $ P $ is solvable (see \cite[Theorem 5]{KMT24indag}).
\end{example}

\begin{example}[Wave Operator]
	Let $ L = {\partial^2}/{\partial t^2} - \eta^2 \sum_{j=1}^{n} {\partial^2}/{\partial x_j^2} $, with $\eta > 0$. The operator $ P $ is $\mathscr{M}$-solvable and $\mathscr{M}$-globally hypoelliptic if and only if one of the following conditions holds:
	\begin{itemize}
		\item[(1)] $|B| < |\text{Im}(A)|$;
		\item[(2)] $|A| = |B|$, $\text{Re}(A) = 0$, and $\eta$ is an irrational non-Liouville number;
		\item[(3)] \eqref{DC_M-condition} holds.
	\end{itemize}
This result is derived from Corollary \ref{coro_GH_ultra}, Proposition \ref{smooth_ultra}, and the smooth solvability established in \cite[Theorem 6]{KMT24indag}.
\end{example}

\begin{example}[Complex Vector Fields]
	As a direct consequence of Theorem \ref{teo_ccoef_gen_ultra}, we can extend the results involving constant-coefficient complex vector fields obtained in \cite{AD,BDM,KMT24indag} to the framework of ultradifferentiable solvability and regularity. Consider
	\begin{equation*}
		L = \dfrac{\partial}{\partial t} + \sum_{j=1}^{n} C_j \dfrac{\partial}{\partial x_j},
	\end{equation*}
	where $C_j \in \mathbb{C}$ for $j = 1, \ldots, n$. Then, $P$ is $\mathscr{M}$-solvable if and only if, for all $\varepsilon > 0$, there exist $C_\varepsilon, \gamma_\varepsilon > 0$ such that
	\begin{equation*} 
		\|\xi\| + |\tau| \geq \gamma_\varepsilon \ \Rightarrow\ |\Delta_{\xi,\tau}| \geq C_{\varepsilon} \,\inf_{j \in \mathbb{N}_0}  \dfrac{m_j \,j!}{\varepsilon^j (1 + \|\xi\| + |\tau|)^j} ,
	\end{equation*}
	with
	\begin{equation*}\label{Delta_Jk_ultra}
		\Delta_{\xi,\tau} = -|\tau + \xi\cdot C|^2 + |A|^2 - |B|^2 - 2i \, \text{Re}(A(\tau + \xi\cdot\bar{C})).
	\end{equation*}
	
	In particular, if $\text{Im}(C) = 0$, then $P$ is $\mathscr{M}$-solvable if and only if one of the following conditions holds:
	\begin{itemize}
		\item[(1)] $ |B| > |A| $;
		\item[(2)] $ |B| < |A| $ and $ \text{Re}(A) \neq 0 $;
		\item[(3)] (DC$_\mathscr{M}$) holds.
	\end{itemize}
	The proof of this particular case follows the same lines as \cite[Theorem 4]{AD}.
\end{example}

\begin{remark}\label{remark1}
The results for the Gevrey classes obtained by Almeida and Dattori da Silva in \cite{AD} are particular cases of those developed in this paper for the Denjoy-Carleman classes of Roumieu type. In fact, the Gevrey spaces are obtained when we consider the weight sequence $\mathscr{M} = \{(j!)^{s-1}\}_{j \in \mathbb{N}_0}$. It is possible to show that
\begin{equation*}
	t^{1/s} - s \,\log\left(\dfrac{1}{1 - s^{-1}}\right) \leq \sup_{j \in \mathbb{N}_0} \log\left(\dfrac{t^j}{(j!)^s}\right) \leq s \,t^{1/s},
\end{equation*}
which allows us to demonstrate that the condition \eqref{DC_M-condition} is equivalent to the (DC$_s$) condition obtained in \cite{AD}.

\end{remark}

\section{A class of operators with variable coefficients}\label{vcoef_ultra}

Motivated by references \cite{BDM} and \cite{AD}, we examine in this section the $\mathscr{M}$-solvability of a class of Vekua-type operators with variable coefficients.

First, we take the following complex vector field L into consideration:
\begin{equation}\label{L_ultra_geral}
	L = \frac{\partial}{\partial t} - \sum_{j=1}^{n} \left( p_j(t) + i q_j(t) \right) \frac{\partial}{\partial x_j},
\end{equation}
with $p_j, q_j \in \mathcal{E}_{\mathscr{M}}(\mathbb{T}^1; \mathbb{R})$ for $j = 1, \dots, n$.

Assume that $L$ satisfies the Nirenberg-Treves condition $(P)$, implying that $L$ is locally solvable (see \cite{BCH_book, NT67cpam}). Furthermore, there exists a function $q \in \mathcal{E}_{\mathscr{M}} (\mathbb{T}^1; \mathbb{R})$ that does not change sign, and constants $\lambda_1, \ldots, \lambda_n \in \mathbb{R}$ such that $q_j(t) = \lambda_j q(t)$ for all $t \in \mathbb{T}^1$ (see \cite{BDG17jfaa}). Consequently, we can rewrite $L$ as:
\begin{equation}\label{L_ultra_1}
	\widetilde{L} = \frac{\partial}{\partial t} - \sum_{j=1}^{n} \left( p_j(t) + i \lambda_j q(t) \right) \frac{\partial}{\partial x_j}.
\end{equation}
and assume, without loss of generality, that $q\geq 0$. 

Now, set $m(t) = (m_1(t), \dots, m_n(t))$ where
\begin{align*}
	m_j(t) = \int_{0}^{t} (p_j(\tau) - p_{0j}) \, d\tau,  \quad p_{0j} = \frac{1}{2\pi} \int_{0}^{2\pi} p_j(\tau) \, d\tau,  \text{ for } \ j = 1, \dots, n,
\end{align*}
and define the operator $T : \mathscr{D}_{\mathscr{M}}'(\mathbb{T}^{n+1}) \to \mathscr{D}_{\mathscr{M}}'(\mathbb{T}^{n+1})$ by
\begin{equation} \label{automorphism}
	Tu(t, x) = \sum_{\xi \in \mathbb{Z}^n} \widehat{u}(t, \xi) e^{-i m(t) \cdot\xi} e^{i \xi \cdot x},
\end{equation}
where the ultradistribution $u(t, x) = \sum_{\xi \in \mathbb{Z}^n} \widehat{u}(t, \xi) e^{i \xi \cdot x} \in \mathscr{D}_{\mathscr{M}}'(\mathbb{T}^{n+1})$ is expressed in its partial Fourier series representation.

By Theorem 5.12 of \cite{artigo_bruno}, the operator $T$ and its restriction $T|_{\mathcal{E}_{\mathscr{M}}} : \mathcal{E}_{\mathscr{M}} (\mathbb{T}^{n+1}) \to \mathcal{E}_{\mathscr{M}} (\mathbb{T}^{n+1})$ are automorphisms that satisfy
\begin{equation*}
	T\widetilde{L}T^{-1} = \dfrac{\partial}{\partial t} - \sum_{j=1}^n (p_{0j} + i \lambda_j q(t)) \dfrac{\partial}{\partial x_j}.
\end{equation*}

\subsection{The Vekua-type operator} \

Consider the operator $P : \mathcal{E}_{\mathscr{M}} (\mathbb{T}^{n+1}) \to \mathcal{E}_{\mathscr{M}} (\mathbb{T}^{n+1})$ given by
\begin{equation}\label{P_cv_Tn+1_ultra} 
	Pu = Lu - (s(t) + i\delta q(t))u - \alpha q(t) \bar{u},
\end{equation}
with
\begin{equation}\label{L_ultra}
	L = \dfrac{\partial}{\partial t} - \sum_{j=1}^n (p_{0j} + i \lambda_j q(t)) \dfrac{\partial}{\partial x_j} ,
\end{equation}
where the constants $\alpha \in \mathbb{C} \setminus \{0\}$ and $\delta\in \mathbb{R}$, the vectors $\lambda=(\lambda_1, \ldots,\lambda_n)$ and $p_0=(p_{01},\ldots  p_{0n})$  are in $\mathbb{R}^n$, and the functions  $q,s \in \mathcal{E}_{\mathscr{M}} (\mathbb{T}^1; \mathbb{R})$. Recall that $q$ does not change sign, and we can assume $q\geq 0$.  Furthermore, we define
\begin{align*}
	& A_0 \doteq s_0+i\delta q_0, \ 
	B_0 \doteq \alpha q_0, \text{ and } \
	C_0 \doteq 2\pi p_0+i\lambda q_0. \\
	& \text{where } \ q_0 \doteq \int_0^{2\pi}q(\tau)\,d\tau > 0  \text{ and } 
	s_0 \doteq \int_0^{2\pi}s(\tau)\,d\tau,
\end{align*}

\begin{theorem}\label{teo_varcoef_1_Tn+1_ultra}
	Assume that the following conditions hold:
	\begin{itemize}
		\item[(I)] $|\alpha| \neq |\delta|$;
		\item[(II)] There is no $(\xi,\tau) \in \mathbb{Z}^{n+1}$ solution to
		\begin{equation*} 
			\begin{cases}
				\text{Re}(A_0(2\pi \tau + \xi \cdot \overline{C_0})) = 0, \\
				|2\pi \tau + \xi \cdot C_0|^2 = |A_0|^2 - |B_0|^2;
			\end{cases}
		\end{equation*}
		\item[(III)] For all $\varepsilon > 0$, there exist $C_\varepsilon, \gamma_\varepsilon > 0$ such that if $\xi \in \mathbb{Z}^{n}$ and $\|\xi\| \geq \gamma_\varepsilon$, then
		\begin{equation*} 
			\min \left\{ |e^{-\rho_\xi q_0} - e^{s_0 + i 2\pi \xi\cdot p_0}|, |1 - e^{-\rho_\xi q_0 + s_0 + i 2\pi \xi\cdot p_0}| \right\} \geq C_\varepsilon \,\inf_{j \in \mathbb{N}_0} \frac{m_j \,j!}{\varepsilon^j (1 + \|\xi\|)^j},
		\end{equation*}
		where $\rho_\xi \in \{\pm\sqrt{(\lambda \cdot\xi - i \delta)^2 + |\alpha|^2}\}$ we choose to satisfy $\text{Re}(\rho_\xi) \geq 0$.
	\end{itemize}	
	Then, for all $f\in\mathcal{E}_{\mathscr{M}}(\mathbb{T}^{n+1})$, exists $u\in\mathcal{E}_{\mathscr{M}}(\mathbb{T}^{n+1})$ such that $Pu=f$.
\end{theorem}

\begin{remark}
The proof of this result is an adapted version of the proof of Theorem 7 presented in \cite{BDM}, tailored to our case. Due to the intricate nature of ultradifferentiability, this proof is significantly more extensive and challenging to follow. To assist the reader in understanding, we have chosen to retain many intermediate steps rather than omitting them, as usual.
\end{remark}

\begin{proof}
Given $u \in \mathscr{D}_{\mathscr{M}}'(\mathbb{T}^{n+1})$ and $f \in \mathcal{E}_{\mathscr{M}} (\mathbb{T}^{n+1}),$ consider their partial Fourier series representations:
	$$u(x,t)=\displaystyle\sum_{\xi\in\mathbb{Z}^n} \widehat u(t,\xi)e^{i\xi\cdot x}\quad\text{ and }\quad f(x,t)=\displaystyle\sum_{\xi\in\mathbb{Z}^n} \widehat f(t,\xi)e^{i\xi\cdot x}.$$
	
	If $Pu=f$, then following the procedures in \cite{BDM} and \cite{AD}, conditions (I) and (II) give us
	\begin{equation*} 
		\widehat u(t,\xi) = \alpha e^{i(\xi\cdot p_0) t+S(t)}(z_{1\xi}(t)+z_{2\xi}(t)),\ \ \xi\in\mathbb{Z}^n,
	\end{equation*}
	with
	\begin{align*}
		z_{1\xi}(t) = & - \int_t^{2\pi}\!\! e^{\rho_\xi(\widetilde{Q}(t)-\widetilde{Q}(\sigma))}e^{-i(\xi\cdot p_0) \sigma-S(\sigma)}G_{1\xi}(\sigma)\,d\sigma\\
		& \quad \ +\, e^{\rho_\xi\widetilde{Q}(t)}\!\!\displaystyle\int_0^{2\pi}\!\dfrac{e^{-\rho_\xi(q_0+\widetilde{Q}(\sigma))}e^{-i(\xi\cdot p_0)\sigma-S(\sigma)}}{e^{-\rho_\xi q_0}-e^{i(\xi\cdot p_0)2\pi+s_0}}G_{1\xi}(\sigma)\,d\sigma,
	\end{align*}
	\begin{align*}
		z_{2\xi}(t) = & \displaystyle\int_0^t e^{\rho_\xi(Q(\sigma)-Q(t))}e^{-i(\xi\cdot p_0) \sigma-S(\sigma)}G_{2\xi}(\sigma)\,d\sigma\\
		& \qquad  +\, e^{-\rho_\xi Q(t)}\!\!\displaystyle\int_0^{2\pi} \dfrac{ e^{\rho_\xi( Q(\sigma)-q_0)}e^{-i(\xi\cdot p_0)\sigma-S(\sigma)}}{1-e^{-\rho_\xi q_0+i(\xi\cdot p_0)2\pi+s_0}}G_{2\xi}(\sigma)\,d\sigma,
	\end{align*}
	where
	\begin{equation*}
		Q(t)=\int_0^t q(\tau)\,d\tau, \ \widetilde{Q}(t)=-\int_t^{2\pi}q(\tau)\,d\tau, \ S(t)=\int_0^{t}s(\tau)\,d\tau,
	\end{equation*}
	and
	\begin{equation*}
		G_\xi(t)=\begin{bmatrix} G_{1\xi}(t)\\ G_{2\xi}(t)\end{bmatrix}=\dfrac{-1}{2\alpha\rho_\xi}\begin{bmatrix}(\lambda\cdot\xi-i\delta)-\rho_\xi & -\alpha \\ -(\lambda\cdot\xi-i\delta)+\rho_\xi & \alpha \end{bmatrix}\cdot\begin{bmatrix} \widehat f(t,\xi)\\ \overline{\widehat f(t,-\xi)}\end{bmatrix}.
	\end{equation*}	

	Given that $\rho_\xi \in \{\pm\sqrt{(\lambda \cdot\xi - i \delta)^2 + |\alpha|^2}\}$ and $\text{Re}(\rho_\xi) > 0$, we have
	\begin{equation*}
		|\rho_\xi|^4 = (\lambda\cdot\xi)^4 + 2(|\alpha|^2 + \delta^2)(\lambda\cdot\xi)^2 + (|\alpha|^2 - \delta^2)^2 \leq C'(1+|\xi|)^4,
	\end{equation*}
	since $(\lambda\cdot\xi)^4 + 2(|\alpha|^2+\delta^2)(\lambda\cdot\xi)^2 + (|\alpha|^2-\delta^2)^2$ is a  fourth-degree polynomial. Therefore, there exists a positive constant $C$ such that
	\begin{equation}\label{est_rJ_ultra_1}
		|\rho_\xi| \leq C(1+|\xi|), \quad \xi \in \mathbb{Z}^n.
	\end{equation}

	Moreover, we have that
	\begin{equation*}
		|\rho_\xi|^4 = (\lambda\cdot\xi)^4 + 2(|\alpha|^2 + \delta^2)(\lambda\cdot\xi)^2 + (|\alpha|^2 + \delta^2)^2 \geq (|\alpha|^2 + \delta^2)^2 > 0, \quad \xi \in \mathbb{Z}^n.
	\end{equation*}
	Consequently, there exists a positive constant $C_\rho$ such that
	\begin{equation*}
		|\rho_\xi| \geq C_{\rho}, \quad \xi \in \mathbb{Z}^n.
	\end{equation*}
	In particular,
	\begin{equation}\label{est_rJ_ultra_2}
		0 < |\rho_\xi|^{-1} \leq C_\rho^{-1}, \quad \xi \in \mathbb{Z}^n.
	\end{equation}
	
	Based on (\ref{est_rJ_ultra_2}), it follows that the decay of $G_{1\xi}$ and $G_{2\xi}$ is the same as that of $\widehat{f}(\cdot,\xi)$. Since $f\in\mathcal{E}_{\mathscr{M}}(\mathbb{T}^{n+1})$, there exist positive constants $C_G$, $h_{G_1}$, and $h_{G_2}$ such that for all $N\in\mathbb{N}_0$, the following inequality holds:
	\begin{equation*}
		\left|\dfrac{d^N}{dt^N}G_{k\xi}(t)\right| \leq C_G\, h_{G_1}^{N}\,  m_N\, N!  \inf_{j\in\mathbb{N}_0} \dfrac{m_j\,j!}{h_{G_2}^j(1+\|\xi\|)^j},\quad \text{for } k=1,2.
	\end{equation*}
	
	Given $N\in\mathbb{N}_0$, we have
	\begin{equation*}
		\dfrac{d^{N}}{dt^{N}}e^{i(\xi\cdot p_0)t} = i^N(\xi\cdot p_0)^N e^{i(\xi\cdot p_0)t},
	\end{equation*}
	which yields
	\begin{equation*}
		\left|\dfrac{d^{N}}{dt^{N}}e^{i(\xi\cdot p_0)t}\right| \leq h_{p}^N(1+|\xi|)^N |e^{i(\xi\cdot p_0)t}| = h_p^N(1+|\xi|)^N
	\end{equation*}
	for some constant $h_p>0$.
	
	Similarly, we have
	\begin{equation*}
		\left|\dfrac{d^{N}}{dt^{N}}e^{-i(\xi\cdot p_0)t}\right| \leq h_p^N(1+|\xi|)^N.
	\end{equation*}

	Now, employing de Faà di Bruno's Formula, for $N\in\mathbb{N}_0$, we obtain
	\begin{align*}
		\dfrac{d^{N}}{dt^{N}} e^{\rho_\xi\widetilde{Q}(t)}\ \ =\ \ & e^{\rho_\xi\widetilde{Q}(t)}\!\!\!\!\displaystyle\sum_{\gamma\in\Delta(N)}\!\!\dfrac{N!}{\gamma!}\prod_{\ell=1}^{N}\left(\dfrac{1}{\ell!}\cdot\dfrac{d^{\ell}}{dt^{\ell}}\rho_\xi\widetilde{Q}(t)\right)^{\gamma_\ell}\\
		=\ \ & e^{\rho_\xi\widetilde{Q}(t)}\!\!\!\!\displaystyle\sum_{\gamma\in\Delta(N)}\!\!\dfrac{N!}{\gamma!}\prod_{\ell=1}^{N}\left(\dfrac{-\rho_\xi}{\ell!}\cdot\dfrac{d^{\ell-1}}{dt^{\ell-1}}q(t)\right)^{\gamma_\ell}.
	\end{align*}
		
	Since $q\in\mathcal{E}_{\mathscr{M}}(\mathbb{T}^{n+1})$, exists $C_{q0},h_{q0}>0$ such that
	\begin{equation*} 
		\left|\dfrac{d^{k}}{dt^{k}}q(t)\right|\leq C_{q0}\, h_{q0}^{k}\, m_{k}\, k!,\quad k\in\mathbb{N}_0.
	\end{equation*}
	
	Then,
	\begin{align*}
		\left| \dfrac{d^{N}}{dt^{N}} e^{\rho_\xi\widetilde{Q}(t)} \right|\ \ \leq\ \ & |e^{\rho_\xi\widetilde{Q}(t)}| \! \sum_{\gamma\in\Delta(N)}\!\!\dfrac{N!}{\gamma!}|\rho_\xi|^{|\gamma|}\prod_{\ell=1}^{N}\left(\dfrac{1}{\ell!}\, C_{q0}\, h_{q0}^{\ell-1}\, m_{\ell-1}\, (\ell-1)! \right)^{\gamma_\ell}\\
		\leq\ \ & |e^{\rho_\xi\widetilde{Q}(t)}|\! \sum_{\gamma\in\Delta(N)} \! \dfrac{N!}{\gamma!} C^{|\gamma|} (1+\|\xi\|)^{|\gamma|}\prod_{\ell=1}^{N} \left( \dfrac{1}{\ell}\, C_{q0}\, h_{q0}^{\ell}\, m_{\ell} \right)^{\gamma_\ell}\\
		\leq\ \ & |e^{\rho_\xi\widetilde{Q}(t)}|h_{q0}^N\! \sum_{\gamma\in\Delta(N)} \! \dfrac{1}{\gamma!}(C  C_{q0})^{|\gamma|} (1+\|\xi\|)^{|\gamma|}\prod_{\ell=1}^{N}\left(\dfrac{1}{\ell}\,  m_{\ell}\right)^{\gamma_\ell}\\
		\leq\ \ & |e^{\rho_\xi\widetilde{Q}(t)}|h_{q0}^N \! \sum_{\gamma\in\Delta(N)} \!\dfrac{1}{\gamma!} C_q^{|\gamma|} (1+\|\xi\|)^{|\gamma|}\dfrac{m_{N}}{m_{|\gamma|}}.
	\end{align*}
	
	Notice that we have used the following facts: $\mathscr{M}$ is non-decreasing, Lemma \ref{lema_prod_m} in \pageref{lema_prod_m}, the estimate (\ref{est_rJ_ultra_1}), and that \begin{equation*}
		\prod_{\ell=1}^N\left(\dfrac{1}{\ell}\right)^{\gamma_\ell}\leq \prod_{\ell=1}^N\left(\dfrac{1}{\ell}\right) = \dfrac{1}{N!},
	\end{equation*} 
	Additionally, in the last inequality, we denoted $C_q=CC_{q0}$.
	
	Since $s\in\mathcal{E}_{\mathscr{M}}(\mathbb{T}^{n+1})$, there exist $C_{s0},h_{s0}>0$ such that,
	\begin{equation*} 
		\left|\dfrac{d^{k}}{dt^{k}}s(t)\right|\leq C_{s0}\, h_{s0}^{k}\, m_{k}\, k!, \quad k\in\mathbb{N}_0.
	\end{equation*}
	
	Analogously, 
	\begin{align*} 
		\left|\dfrac{d^{N}}{dt^{N}} e^{S(t)} \right|  \leq |e^{S(t)}|h_{s0}^N\!\!\! \sum_{\gamma\in\Delta(N)} \!\!\dfrac{1}{\gamma!}(C_{s0})^{|\gamma|}\dfrac{m_{N}}{m_{|\gamma|}},  \quad N\in\mathbb{N}_0.
	\end{align*}
	
	For $e^{-i(\xi\cdot p_0)t}$, $e^{-\rho_\xi\widetilde{Q}(t)}$ and $e^{-S(t)}$, the arguments are similar. Then, we have the following estimates, for all $N\in\mathbb{N}_0$,
	\begin{align}
		&\left|\dfrac{d^{N}}{dt^{N}}e^{\pm i\xi\cdot p_0t} \ \right|  \leq h_p^N(1+\|\xi\|)^N; \label{D_eiJ}\\[1mm]
		&\left|\dfrac{d^{N}}{dt^{N}} e^{\pm\rho_\xi\widetilde{Q}(t)} \right|  \leq |e^{\pm\rho_\xi\widetilde{Q}(t)}|h_{q0}^N\! \sum_{\gamma\in\Delta(N)} \! \dfrac{1}{\gamma!} C_q^{|\gamma|}(1+\|\xi\|)^{|\gamma|} \dfrac{m_{N}}{m_{|\gamma|}}; \label{D_rJQ}  \\[1mm]
		& \left|\dfrac{d^{N}}{dt^{N}} e^{\pm S(t)} \right|  \leq |e^{\pm S(t)}|h_{s0}^{N}\!\!\!\sum_{\gamma\in\Delta(N)}\!\!\dfrac{1}{\gamma!}(C_{s0})^{|\gamma|}\dfrac{m_{N}}{m_{|\gamma|}}. \label{D_eS}
	\end{align}

	After obtaining all the above estimates, we can now study the decay of the derivatives of $e^{i\xi\cdot p_0t+S(t)}z_{1\xi}(t)$. For $N\in\mathbb{N}_0$, we have
	\begin{align*}
		& \dfrac{d^N}{dt^N}\left(e^{i\xi\cdot p_0 t+S(t)}e^{\rho_\xi\widetilde{Q}(t)}\displaystyle\int_t^{2\pi}e^{-\rho_\xi\widetilde{Q}(\sigma)}e^{-i(\xi\cdot p_0)\sigma-S(\sigma)}G_{1\xi}(\sigma)\,d\sigma\right)\\
		=\ \ & \displaystyle\sum_{k=0}^N\binom{N}{k}\dfrac{d^{N-k}}{dt^{N-k}}\left(e^{i(\xi\cdot p_0)t+S(t)}e^{\rho_\xi\widetilde{Q}(t)}\right)  \dfrac{d^{k}}{dt^{k}} \left(\int_t^{2\pi}e^{-\rho_\xi\widetilde{Q}(\sigma)}e^{-i(\xi\cdot p_0)\sigma-S(\sigma)}G_{1\xi}(\sigma)\,d\sigma\right)\\
		=\ \ & \dfrac{d^{N}}{dt^{N}}\left(e^{i(\xi\cdot p_0)t+S(t)}e^{\rho_\xi\widetilde{Q}(t)}\right) \left( \int_t^{2\pi}e^{-\rho_\xi\widetilde{Q}(\sigma)}e^{-i(\xi\cdot p_0)\sigma-S(\sigma)}G_{1\xi}(\sigma)\,d\sigma\right)\\
		& + \displaystyle\sum_{k=1}^N\binom{N}{k}\dfrac{d^{N-k}}{dt^{N-k}}\left(e^{i(\xi\cdot p_0)t+S(t)}e^{\rho_\xi\widetilde{Q}(t)}\right) \dfrac{d^{k-1}}{dt^{k-1}}\left(e^{-\rho_\xi\widetilde{Q}(t)}e^{-i(\xi\cdot p_0)t-S(t)}G_{1\xi}(t)\right).
	\end{align*}
	
	Given $k=1,\dots,N_1$ and $r=0,\dots,k-1$, it follows from (\ref{D_rJQ}) and (\ref{D_eS}) that
	\begin{align*}
		\left|\dfrac{d^{r}}{dt^{r}} \left(e^{-\rho_\xi\widetilde{Q}(t)-S(t)}\right)\right|\ \ 
		\leq\ \ & \displaystyle\sum_{w=0}^{r}\binom{r}{w}\left|\dfrac{d^{w}}{dt^{w}}\left(e^{-\rho_\xi\widetilde{Q}(t)}\right)\right|\,\left|\dfrac{d^{r-w}}{dt^{r-w}} \left(e^{-S(t)}\right)\right| \\
		\leq\ \ & |e^{-\rho_\xi\widetilde{Q}(t)-S(t)}|\displaystyle\sum_{w=0}^{r}\binom{r}{w} \left[ h_{q0}^w \!\!\! \displaystyle\sum_{\gamma\in\Delta(w)} \!\! \dfrac{1}{\gamma!}(C_q)^{|\gamma|}(1+\|\xi\|)^{|\gamma|}\dfrac{m_{w}}{m_{|\gamma|}}\right]\\
		& \times \left[h_{s0}^{r-w}\! \sum_{\eta\in\Delta(r-w)} \!\dfrac{1}{\eta!} (C_{s0})^{|\eta|} \dfrac{m_{r-w}}{m_{|\eta|}}\right]
	\end{align*}
	
	Denoting $k_r\doteq k-1-r$, we have
	\begin{align*}
		& \left|\dfrac{d^{k_r}}{dt^{k_r}}\left(e^{-i(\xi\cdot p_0)t}G_{1\xi}(t)\right)\right|\ \leq\  \displaystyle\sum_{w=0}^{k_r}\binom{k_r}{w}\left|\dfrac{d^{k_r-w}}{dt^{k_r-w}}(G_{1\xi}(t))\right| \left|\dfrac{d^{w}}{dt^{w}}e^{-i(\xi\cdot p_0)t}\right|\\
		\leq\ \ & \displaystyle\sum_{w=0}^{k_r}\binom{k_r}{w}\, h_p^w(1+\|\xi\|)^w C_Gh_{G_1}^{k_r-w}m_{k_r-w}(k_r-w)! \inf_{j\in\mathbb{N}_0} \dfrac{m_j\, j!}{h_{G_2}^j(1+\|\xi\|)^j}\\
		\leq\ \ &  \displaystyle\sum_{w=0}^{k_r}\binom{k_r}{w}\, h_p^w(1+\|\xi\|)^w C_Gh_{G_1}^{k_r-w}m_{k_r-w}(k_r-w)!\, \left(\inf_{j\in\mathbb{N}_0}\dfrac{m_j\, j!}{(h_{G_2}/H)^j(1+\|\xi\|)^j}\right)^2\\
		\leq\ \ & \displaystyle\sum_{w=0}^{k_r}\binom{k_r}{w}\, h_p^w(1+\|\xi\|)^w C_G\, h_{G_1}^{k_r-w}m_{k_r-w}(k_r-w)!\, \dfrac{m_w\, w!}{(h_{G_2}/H)^w(1+\|\xi\|)^w}\\
		& \times \displaystyle\inf_{j\in\mathbb{N}_0}\dfrac{m_j\, j!}{(h_{G_2}/H)^j(1+\|\xi\|)^j}\\
		\leq\ \ & C_G\,h_{G0}^{k_r}\displaystyle\sum_{w=0}^{k_r}\binom{k_r}{w}\,  m_{k_r-w} m_{w} (k_r-w)! w!\, \inf_{j\in\mathbb{N}_0}\dfrac{m_j\, j!}{\delta_1^j(1+\|\xi\|)^j}\\
		\leq\ \ & C_G\,h_{G0}^{k_r}\displaystyle\sum_{w=0}^{k_r}\binom{k_r}{w}\,  m_{k_r}k_r!\, \inf_{j\in\mathbb{N}_0}\dfrac{m_j\, j!}{\delta_1^j(1+\|\xi\|)^j}\\
		\leq\ \ & C_G\,h_G^{k_r}\, m_{k_r}k_r!\,\displaystyle\inf_{j\in\mathbb{N}_0}\dfrac{m_j\,j!}{\delta_1^j(1+\|\xi\|)^j},
	\end{align*}
	with $\delta_1=h_{G_2}/H$, $h_{G0} = \max\{h_p,h_{G_1}\}$ and $h_G=2h_{G0}$, noticing that
	\begin{equation*} 
		\displaystyle\sum_{w=0}^{k_r}\binom{k_r}{w}=2^{k_r}.
	\end{equation*}
	
	Then, given $k=1\dots,N-1$, using Lemmas \ref{lema2.25}, \ref{lema_sum} and \ref{lema_prod_m} and the estimates (\ref{D_eiJ}), (\ref{D_rJQ}) and (\ref{D_eS}), we have
	\begin{align*}
		& \left| \dfrac{d^{k-1}}{dt^{k-1}} \left(e^{-\rho_\xi\widetilde{Q}(t)}e^{-i(\xi\cdot p_0)t-S(t)}G_{1\xi}(t) \right)\right|\\
		\leq\ \ & \displaystyle\sum_{r=0}^{k-1} \binom{k-1}{r} \left|\dfrac{d^{r}}{dt^{r}} \left(e^{-\rho_\xi\widetilde{Q}(t)-S(t)}\right)\right|\left|\dfrac{d^{k-1-r}}{dt^{k-1-r}}\left(e^{-i(\xi\cdot p_0)t}G_{1\xi}(t)\right)\right|\\
		\leq\ \ & |e^{-\rho_\xi\widetilde{Q}(t)-S(t)}|\displaystyle\sum_{r=0}^{k-1} \binom{k-1}{r}\Bigg[ \displaystyle\sum_{w=0}^{r}\binom{r}{w} \left[ h_{q0}^w \!\!\! \displaystyle\sum_{\gamma\in\Delta(w)} \!\! \dfrac{1}{\gamma!}(C_q)^{|\gamma|}(1+\|\xi\|)^{|\gamma|}\dfrac{m_{w}}{m_{|\gamma|}}\right]\\
		& \times \left[h_{s0}^{r-w}\!\!\!\displaystyle\sum_{\eta\in\Delta(r-w)}\!\!\dfrac{1}{\eta!}(C_{s0})^{|\eta|}\dfrac{m_{r-w}}{m_{|\eta|}}\right] \,C_G\,h_G^{k_r}\, m_{k_r}k_r!\,\displaystyle\inf_{j\in\mathbb{N}_0}\dfrac{m_j\,j!}{\delta_1^j(1+\|\xi\|)^j}\Bigg]\\
		\leq\ \   & |e^{-\rho_\xi\widetilde{Q}(t)-S(t)}|\displaystyle\sum_{r=0}^{k-1} \binom{k-1}{r}\Bigg[ \displaystyle\sum_{w=0}^{r}\binom{r}{w} \left[ h_{q0}^w \!\!\! \displaystyle\sum_{\gamma\in\Delta(w)} \!\! \dfrac{1}{\gamma!}(C_q)^{|\gamma|}(1+\|\xi\|)^{|\gamma|}\dfrac{m_{w}}{m_{|\gamma|}}\right]\\
		& \times \left[h_{s0}^{r-w}\!\!\!\displaystyle\sum_{\eta\in\Delta(r-w)}\!\!\dfrac{1}{\eta!}(C_{s0})^{|\eta|}\dfrac{m_{r-w}}{m_{|\eta|}}\right] \,C_G\,h_G^{k_r}\, m_{k_r}k_r!\,\displaystyle\inf_{j\in\mathbb{N}_0}\dfrac{m_j\,j!}{(\delta_1/H)^j(1+\|\xi\|)^j}\\
		& \times \left(\displaystyle\inf_{j\in\mathbb{N}_0}\dfrac{m_j\,j!}{(\delta_1/H^2)^j(1+\|\xi\|)^j}\right)^2\Bigg]\\
		\leq\ \  & |e^{-\rho_\xi\widetilde{Q}(t)-S(t)}|\displaystyle\sum_{r=0}^{k-1} \binom{k-1}{r}\Bigg[ \displaystyle\sum_{w=0}^{r}\binom{r}{w} \left[ h_{q0}^w \!\!\! \displaystyle\sum_{\gamma\in\Delta(w)} \!\! \dfrac{1}{\gamma!}(C_q)^{|\gamma|}(1+\|\xi\|)^{|\gamma|}\dfrac{m_{w}}{m_{|\gamma|}}\right]\\
		& \times \dfrac{m_{|\gamma|}\,|\gamma|!}{\delta_3^{|\gamma|}(1+\|\xi\|)^{|\gamma|}}\,\left[h_{s0}^{r-w}\!\!\!\displaystyle\sum_{\eta\in\Delta(r-w)}\!\!\dfrac{1}{\eta!}(C_{s0})^{|\eta|}\dfrac{m_{r-w}}{m_{|\eta|}}\right]\,\dfrac{m_{|\eta|}\,|\eta|!}{\delta_3^{|\eta|}(1+\|\xi\|)^{|\eta|}}\\
		& \times C_G\,h_G^{k_r}\, m_{k_r}\,k_r!\,\displaystyle\inf_{j\in\mathbb{N}_0}\dfrac{m_j\,j!}{\delta_2^j(1+\|\xi\|)^j}\Bigg]\\
		\leq\ \  & |e^{-\rho_\xi\widetilde{Q}(t)-S(t)}|C_Gh_0^{k-1}\displaystyle\sum_{r=0}^{k-1} \binom{k-1}{r}\Bigg[ \displaystyle\sum_{w=0}^{r}\binom{r}{w} m_wm_{r-w}\left[ \displaystyle\sum_{\gamma\in\Delta(w)} \!\! \dfrac{|\gamma|!}{\gamma!}\left(\dfrac{C_q}{\delta_3}\right)^{|\gamma|}\right]\\
		& \times \left[\displaystyle\sum_{\eta\in\Delta(r-w)}\!\!\dfrac{|\eta|!}{\eta!}\left(\dfrac{C_{s0}}{\delta_3}\right)^{|\eta|}\right] \, m_{k_r}\,k_r!\,\displaystyle\inf_{j\in\mathbb{N}_0}\dfrac{m_j\,j!}{\delta_2^j(1+\|\xi\|)^j}\Bigg]\\
		\leq\ \  & |e^{-\rho_\xi\widetilde{Q}(t)-S(t)}|C_Gh_0^{k-1}\displaystyle\sum_{r=0}^{k-1} \binom{k-1}{r} m_rm_{k-1-r}\Bigg[ \displaystyle\sum_{w=0}^{r}\binom{r}{w} C_{q1}(1+C_{q1})^{w-1}\\
		& \times C_{s1}(1+C_{s1})^{r-w-1} \,(k-1-r)!\,\displaystyle\inf_{j\in\mathbb{N}_0}\dfrac{m_j\,j!}{\delta_2^j(1+\|\xi\|)^j}\Bigg]\\
		\leq\ \  & |e^{-\rho_\xi\widetilde{Q}(t)-S(t)}|C_Gh_0^{k-1}m_{k-1}\displaystyle\sum_{r=0}^{k-1} \binom{k-1}{r}\Bigg[ \displaystyle\sum_{w=0}^{r}\binom{r}{w} C_{q2}^{w}C_{s2}^{r-w}(k-1-r)!\\
		& \times\displaystyle \inf_{j\in\mathbb{N}_0}\dfrac{m_j\,j!}{\delta_2^j(1+\|\xi\|)^j}\Bigg]\\
		\leq \ \ &  |e^{-\rho_\xi\widetilde{Q}(t)-S(t)}|\,C_{G0}\,h_{00}^{k-1}\,m_{k-1}\,(k-1)!\,\displaystyle \inf_{j\in\mathbb{N}_0}\dfrac{m_j\,j!}{\delta_2^j(1+\|\xi\|)^j},
	\end{align*}
	with $\delta_2=\delta_1/H$, $\delta_3=\delta_1/H^2$, $h_0=\max\{h_{q0},h_{s0}\}$, $C_{q1}=C_q/\delta_3$, $C_{s1}=C_{s0}/\delta_3$, $C_{q_2}=(1+C_{q1})$, $C_{s2}=1+C_{s1}$, $h_{00}=2h_0$ and $C_{G0}=2C_G$, observing that
	\begin{equation*} 
		\displaystyle\sum_{r=0}^{k-1} \binom{k-1}{r} \displaystyle\sum_{w=0}^{r}\binom{r}{w}\ =\ 2^k-\dfrac{1}{2}\ \leq\ 2^k\ =\ 2\cdot 2^{k-1}.
	\end{equation*}
	
	Given $k=1,\dots,N$, it follows from the estimates (\ref{D_rJQ}), (\ref{D_rJQ}) and (\ref{D_eS}) that
	\begin{align*}
		& \left|\dfrac{d^{N-k}}{dt^{N-k}}(e^{i(\xi\cdot p_0)t+S(t)}e^{\rho_\xi \widetilde{Q}(t)})\right|\\
		\leq\ \ & \displaystyle\sum_{r=0}^{N-k}\binom{N-k}{r} \left|\dfrac{d^{r}}{dt^{r}} \left(e^{\rho_\xi\widetilde{Q}(t)+S(t)}\right)\right|\,\left|\dfrac{d^{N-k-r}}{dt^{N-k-r}}e^{i(\xi\cdot p_0)t}\right|\\
		\leq\ \ & \displaystyle\sum_{r=0}^{N-k}\binom{N-k}{r} |e^{\rho_\xi\widetilde{Q}(t)+S(t)}|\displaystyle\sum_{w=0}^{r}\binom{r}{w} \left[ h_{q0}^w \!\!\! \displaystyle\sum_{\gamma\in\Delta(w)} \!\! \dfrac{1}{\gamma!}(C_q)^{|\gamma|}(1+\|\xi\|)^{|\gamma|}\dfrac{m_{w}}{m_{|\gamma|}}\right]\\
		& \times \left[h_{s0}^{r-w}\!\!\!\displaystyle\sum_{\eta\in\Delta(r-w)}\!\!\dfrac{1}{\eta!}(C_{s0})^{|\eta|}\dfrac{m_{r-w}}{m_{|\eta|}}\right]\,h_p^{N-k-r}\,(1+\|\xi\|)^{N-k-r}.
	\end{align*}
	
	Finally, we have
	\begin{align*}
		& \left|\displaystyle\sum_{k=1}^N\binom{N}{k}\dfrac{d^{N-k}}{dt^{N-k}}(e^{i(\xi\cdot p_0)t + S(t)} e^{\rho_\xi\widetilde{Q}(t)})\,\dfrac{d^{k-1}}{dt^{k-1}}\left(e^{-\rho_\xi\widetilde{Q}(t)}e^{-i(\xi\cdot p_0)t - S(t)} G_{1\xi}(t)\right)\right|\\
		\leq\ \ & \displaystyle\sum_{k=1}^N\binom{N}{k}\left|\dfrac{d^{N-k}}{dt^{N-k}}(e^{i(\xi\cdot p_0)t+ S(t)} e^{\rho_\xi\widetilde{Q}(t)})\right|\,\left|\dfrac{d^{k-1}}{dt^{k-1}}\left(e^{-\rho_\xi\widetilde{Q}(t)}e^{-i(\xi\cdot p_0)t-S(t)}G_{1\xi}(t)\right)\right|\\
		\leq\ \ & \displaystyle\sum_{k=1}^N\binom{N}{k}\Bigg[\displaystyle\sum_{r=0}^{N-k}\binom{N-k}{r} |e^{\rho_\xi\widetilde{Q}(t)+S(t)}|\displaystyle\sum_{w=0}^{r}\binom{r}{w} \left[ h_{q0}^w \!\!\! \displaystyle\sum_{\gamma\in\Delta(w)} \!\! \dfrac{1}{\gamma!}(C_q)^{|\gamma|}(1+\|\xi\|)^{|\gamma|}\dfrac{m_{w}}{m_{|\gamma|}}\right]\\
		& \times \left[h_{s0}^{r-w}\!\!\!\displaystyle\sum_{\eta\in\Delta(r-w)}\!\!\dfrac{1}{\eta!}(C_{s0})^{|\eta|}\dfrac{m_{r-w}}{m_{|\eta|}}\right]\,h_p^{N-k-r}\,(1+\|\xi\|)^{N-k-r}\Bigg]\,|e^{-\rho_\xi\widetilde{Q}(t)-S(t)}|\, C_{G0}  \\
		& \times h_{00}^{k-1}\,m_{k-1}\,(k-1)!\,\displaystyle \left[\inf_{j\in\mathbb{N}_0}\left(\dfrac{m_j\,j!}{(\delta_2/H^2)^j(1+\|\xi\|)^j}\right)\right]^4\\
		\leq\ \ & \displaystyle\sum_{k=1}^N\binom{N}{k}\Bigg[\displaystyle\sum_{r=0}^{N-k}\binom{N-k}{r} h_{qs}^{r}h_{p}^{N-k-r}\displaystyle\sum_{w=0}^{r}\binom{r}{w}\Bigg[\displaystyle\sum_{\gamma\in\Delta(w)} \!\! \dfrac{1}{\gamma!}(C_q)^{|\gamma|}(1+\|\xi\|)^{|\gamma|}\dfrac{m_{w}}{m_{|\gamma|}}\\
		& \times \dfrac{m_{|\gamma|}\,|\gamma|!}{\delta_4^{|\gamma|}(1+\|\xi\|)^{|\gamma|}}\Bigg]\,\left[\displaystyle\sum_{\eta\in\Delta(r-w)}\!\!\dfrac{1}{\eta!}(C_{s0})^{|\eta|}\dfrac{m_{r-w}}{m_{|\eta|}}\,\dfrac{m_{|\eta|}\,|\eta|!}{\delta_4^{|\eta|}(1+\|\xi\|)^{|\eta|}}\right]\Bigg]\\
		& \times (1+\|\xi\|)^{N-k-r}\,\dfrac{m_{N-k-r}\,(N-k-r)!}{\delta_4^{N-k-r}(1+\|\xi\|)^{N-k-r}}\,C_{G0}\,h_{00}^{k-1}\,m_{k-1}\,(k-1)!\\
		& \times \displaystyle \inf_{j\in\mathbb{N}_0}\left(\dfrac{m_j\,j!}{\delta_4^j(1+\|\xi\|)^j}\right)\\
		\leq\ \ & C_{G0}\displaystyle\sum_{k=1}^N\binom{N}{k} h_{00}^{k-1}h_{pqs}^{N-k}\Bigg[\displaystyle\sum_{r=0}^{N-k}\binom{N-k}{r}\displaystyle\sum_{w=0}^{r}\binom{r}{w}m_w\Bigg[\displaystyle\sum_{\gamma\in\Delta(w)} \!\! \dfrac{|\gamma|!}{\gamma!}\left(\dfrac{C_q}{\delta_4}\right)^{|\gamma|}\Bigg] m_{r-w}\\
		& \times \Bigg[\displaystyle\sum_{\eta\in\Delta(r-w)}\!\!\dfrac{1}{\eta!}\left(\dfrac{C_{s0}}{\delta_4}\right)^{|\eta|}\Bigg] m_{N-k-r}\,\dfrac{1}{\delta_4^{N-k-r}}\,(N-k-r)!\Bigg]\,m_{k-1}\,(k-1)!\\
		& \times \displaystyle \inf_{j\in\mathbb{N}_0}\dfrac{m_j\,j!}{\delta_4^j(1+\|\xi\|)^j}\\
		\leq\ \ & C_{G0}\,h_{00}^{N} \displaystyle\sum_{k=1}^N\binom{N}{k} m_{k-1} \displaystyle\sum_{r=0}^{N-k}\binom{N-k}{r} m_{N-k-r}m_r \displaystyle\sum_{w=0}^{r}\binom{r}{w}[C_{q3}(1+C_{q3})^{w-1}]\\ & \times [C_{s3}(1+C_{s3})^{r-w-1}](N-r-1)!\,\inf_{j\in\mathbb{N}_0}\dfrac{m_j\,j!}{\delta_4^j(1+\|\xi\|)^j}\\
		\leq\ \ & C_{G0}\,h_{00}^{N} \displaystyle\sum_{k=1}^N\binom{N}{k} m_{k-1} \displaystyle\sum_{r=0}^{N-k}\binom{N-k}{r} m_{N-k-r}m_r \displaystyle\sum_{w=0}^{r}\binom{r}{w}C_{q4}^{w}\,C_{s4}^{r-w}\,(N-r-1)!\\ & \times \,\inf_{j\in\mathbb{N}_0}\dfrac{m_j\,j!}{\delta_4^j(1+\|\xi\|)^j}\\
		\leq\ \ & C\,h^N\,m_N\,N!\, \displaystyle \inf_{j\in\mathbb{N}_0}\dfrac{m_j\,j!}{\delta_4^j(1+\|\xi\|)^j},
	\end{align*}
	with $\delta_4=\delta_2/H^2$, $h_{qs}=\max\{h_{q0},h_{s0}\}$, $h_{pqs}=\max\{h_{p},h_{qs}\}$, $C_{q3}=C_q/\delta_4$, $C_{s3}=C_{s_0}/\delta_4$, $C_{q4}=1+C_{q3}$, $C_{s4}=1+C_{s3}$, $h=2\max\{h_{00},h_{pqs},C_{q4},C_{s4}\}$ and $C=2C_{G0}$.
	
	Using the same arguments, we can obtain similar estimates for the derivatives of
	\begin{equation*}\dfrac{d^{N}}{dt^{N}}\left(e^{i(\xi\cdot p_0)t+S(t)}e^{\rho_\xi\widetilde{Q}(t)}\right)\left(\displaystyle\int_t^{2\pi}e^{-\rho_\xi\widetilde{Q}(\sigma)}e^{-i(\xi\cdot p_0)\sigma-S(\sigma)}G_{1\xi}(\sigma)\,d\sigma\right)
	\end{equation*}
	and of
	\begin{equation*}e^{i(\xi\cdot p_0)t+S(t)}e^{\rho_\xi\widetilde{Q}(t)}\!\!\displaystyle\int_0^{2\pi}\!\dfrac{e^{-\rho_\xi(q_0+\widetilde{Q}(\sigma))}e^{-i(\xi\cdot p_0)\sigma-S(\sigma)}}{e^{-\rho_\xi q_0}-e^{i(\xi\cdot p_0)2\pi+s_0}}G_{1\xi}(\sigma)\,d\sigma,
	\end{equation*}
	observing that
	\begin{align*}
		\left|\displaystyle\int_t^{2\pi}\!\!\! e^{-\rho_\xi\widetilde{Q}(\sigma)}e^{-i(\xi\cdot p_0)\sigma-S(\sigma)}G_{1\xi}(\sigma)\,d\sigma\right|\ \ \leq\ \ &  \displaystyle\int_t^{2\pi}|e^{-\rho_\xi\widetilde{Q}(\sigma)-S(\sigma)}G_{1\xi}(\sigma)|\,d\sigma\\
		\leq\ \ & 2\pi\,\displaystyle\sup|e^{-\rho_\xi\widetilde{Q}(t)-S(t)}|\sup|G_{1\xi}(t)|\\
		\leq\ \ & C h_{G_1}\!\,\displaystyle\sup|e^{-\rho_\xi\widetilde{Q}(t)-S(t)}|\,\inf_{j\in\mathbb{N}_0}\dfrac{m_j\,j!}{h_{G_2}^j(1+\|\xi\|)^j}
	\end{align*}
	and that the condition (III) holds.
	
	The study of the decay of the derivatives of $e^{i(\xi\cdot p_0)t+S(t)}z_{2\xi}(t)$ is analogous. Therefore, $u\in\mathcal{E}_{\mathscr{M}}(\mathbb{T}^n)$ and $Pu=f\in\mathcal{E}_{\mathscr{M}}(\mathbb{T}^n)$ by construction.	
\end{proof}

\begin{remark}
	For the bidimensional torus $\mathbb{T}^2$, conditions (I) and (II) combined with $\lambda\neq 0$ imply condition (III) (see \cite{BDM}). However, in general, this implication does not hold (see \cite{AD}, Example 2).
\end{remark}

\begin{theorem}\label{teo_varcoef_2_Tn+1_ultra}
	Let $P$ the Vekua-type operator defined in (\ref{P_cv_Tn+1_ultra}) with $\lambda=0$. Suppose that one of the following conditions holds:
	\begin{itemize}
		\item[(1)] $|B_0|>|A_0|$;
		\item[(2)] $|B_0|\leq |A_0|$, $|\alpha|>|\delta|$ and there is no $(\xi,\tau)\in\mathbb{Z}^{n+1}$ solution of
		\begin{equation}\label{cv2_caso2_Tn+1_ultra}
			\begin{cases}\text{Re}(A_0(\tau+(\xi\cdot p_0)))=0\\ 4\pi^2|\tau+(\xi\cdot p_0)|^2=|A_0|^2-|B_0|^2\end{cases};
		\end{equation} 
		\item[(3)] $|\alpha|<|\delta|$ and $s_0\neq 0$;
		\item[(4)] $|\alpha|<|\delta|$, $s_0=0$, there is no $(\xi,\tau)\in\mathbb{Z}^{n+1}$ satisfying (\ref{cv2_caso2_Tn+1_ultra}) and the following diophantine condition holds: \\
		 For all $\varepsilon>0$, exist $\gamma_\varepsilon,C_\varepsilon>0$ such that $(\xi,\tau)\in\mathbb{Z}^{n+1}$, $\ \|\xi\|\geq\gamma_\varepsilon$ implies
		\begin{equation}\label{DC_Mprime} 
			|2\tau\pi +(\xi\cdot p_0)2\pi-q_0\sqrt{\delta^2-|\alpha|^2}|\geq C_\varepsilon\,\inf_{j\in\mathbb{N}_0}\dfrac{m_j\,j!}{\varepsilon^j(1+\|\xi\|)^j}. \tag{$DC'_\mathscr{M}$}
		\end{equation} 
	\end{itemize}
	
	Then, for all $f\in C^{\infty}(\mathbb{T}^{n+1})$, exists $u\in C^{\infty}(\mathbb{T}^{n+1})$ such that $Pu=f$.
\end{theorem}
\begin{proof}
	If (1), (2), (3), or (4) holds, it is not difficult to see that conditions (I) and (II) in Theorem \ref{teo_varcoef_1_Tn+1_ultra} hold. Moreover, if (1), (2), or (3) holds, then there exists $C > 0$ such that
	\begin{equation*}
		|e^{-\rho_\xi q_0} - e^{s_0 + i2\pi (\xi\cdot p_0)}| \geq C \quad \text{and} \quad |1 - e^{-\rho_\xi q_0 + s_0 + i2\pi (\xi\cdot p_0)}| \geq C, \quad \xi \in \mathbb{Z}^n.
	\end{equation*}
	
	In fact, first notice that $\lambda=0$ gives us that, for all $\xi\in\mathbb{Z} $, we have $\rho_\xi = \sqrt{|\alpha|^2 - \delta^2}\doteq\rho$.
	
	Suppose that (1) holds. Notice that $|A_0|^2 = s_0^2 + |\delta|^2q_0^2$ and $|B_0|^2 = |\alpha|^2q_0^2$. Since $|A_0| < |B_0|$, we have
	$$
	|\alpha|^2q_0 > s_0^2 + |\delta|^2q_0^2 \geq |\delta|^2q_0^2 \Rightarrow |\alpha| > |\delta|,
	$$
	which implies $\rho \in \mathbb{R}_{>0}$. Suppose there exists $\xi \in \mathbb{Z}^n$ such that $e^{-\rho q_0} - e^{i(\xi\cdot p_0)2\pi + s_0} = 0$. Then, there exists $\tau \in \mathbb{Z}$ such that
	$$
	\rho q_0 + i(\xi\cdot p_0)2\pi + s_0 = -i2\pi \tau.
	$$
	
	Comparing the real parts, we obtain $\rho q_0 + s_0 = 0$. Since $\rho = \sqrt{|\alpha|^2 - \delta^2}$, we have
	\begin{align*} 
		q_0\sqrt{|\alpha|^2 - \delta^2} + s_0 = 0 & \Rightarrow (|\alpha|^2 - \delta^2)q_0^2 = s_0^2 \\
		& \Rightarrow |\alpha|^2q_0^2 = s_0^2 + \delta^2 q_0^2 \\
		& \Rightarrow |A_0|^2 = |B_0|^2,
	\end{align*}
	which contradicts (1). Therefore, $e^{-\rho q_0} - e^{i(\xi\cdot p_0)2\pi + s_0} \neq 0$ for all $\xi \in \mathbb{Z}^n$. Analogously, we show that $1 - e^{-\rho q_0 + i(\xi\cdot p_0)2\pi + s_0} \neq 0$ for all $\xi \in \mathbb{Z}^n$. This implies
	$$
	-\rho q_0 \neq i(\xi\cdot p_0)2\pi + s_0 + i2\pi k, \quad \text{for } k \in \mathbb{N}.
	$$
	Notice that, since $|\alpha| > |\delta|$, we have $\rho \in \mathbb{R}$. Then, given $\xi \in \mathbb{Z}^n$, we have
	$$
	|e^{-\rho q_0} - e^{i(\xi\cdot p_0)2\pi}| \geq |e^{-\rho q_0} - e^{s_0}| \quad \text{and} \quad |1 - e^{-\rho q_0 + i(\xi\cdot p_0)2\pi + s_0}| \geq |1 - e^{-\rho q_0 + s_0}|.
	$$
	
	Choosing $C = \min\{|e^{-\rho q_0} - e^{s_0}|, |1 - e^{-\rho q_0 + s_0}|\} > 0$, we have
	$$
	|e^{-\rho q_0} - e^{i(\xi\cdot p_0)2\pi + s_0}| \geq C \quad \text{and} \quad |1 - e^{-\rho q_0 + i(\xi\cdot p_0)2\pi + s_0}| \geq C, \quad \xi \in \mathbb{Z}^n.
	$$
	
	If (2) or (3) holds, the arguments are similar. Notice that
	$$
	C = C \frac{m_0 \cdot 0!}{\varepsilon^0(1+\|\xi\|)^0} \geq C \inf_{j\in\mathbb{N}_0} \frac{m_j \cdot j!}{\varepsilon^j(1+\|\xi\|)^j}
	$$
	for all $\varepsilon > 0$. Proceeding as in Theorem \ref{teo_varcoef_1_Tn+1_ultra}, we show that for all $f \in \mathcal{E}_{\mathscr{M}}(\mathbb{T}^n)$, there exists $u \in \mathcal{E}_{\mathscr{M}}(\mathbb{T}^n)$ such that $Pu = f$.
	
	If (4) holds, we can not prove that there exists $C > 0$ like in the other cases. In this case, we use the arguments of Theorem \ref{teo_varcoef_1_Tn+1_ultra} and the Diophantine condition \eqref{DC_M_2prime} which is equivalent to \eqref{DC_Mprime}, as proved in the following lemma.
\end{proof}

\begin{lemma}\label{dc3dc4_ultra}
	The condition \eqref{DC_Mprime} is equivalent to the following condition:\\
	For all $\varepsilon > 0$, there exist $\gamma_\varepsilon, C_\varepsilon > 0$ such that for $\xi \in \mathbb{Z}^n$, $\|\xi\| \geq \gamma_\varepsilon$, we have
	\begin{equation} \label{DC_M_2prime}
		\left| e^{i((\xi\cdot p_0) 2\pi-q_0\sqrt{\delta^2-|\alpha|^2})} - 1 \right| \geq C_\varepsilon \inf_{j \in \mathbb{N}_0} \frac{m_j \cdot j!}{\varepsilon^j (1 + \|\xi\|)^j}. \tag{$DC''_\mathscr{M}$}
	\end{equation} 
\end{lemma}
\begin{proof}
	This proof is an adaptation of the proof of Lemma 13 in \cite{BDM}. Suppose that \eqref{DC_Mprime} does not hold. Then, there exists $\varepsilon > 0$ such that for each $\ell \in \mathbb{N}$, there exists $(\xi_\ell, \tau_\ell) \in \mathbb{Z}^{n+1}$ with $\|\xi_\ell\| \geq \ell$ and
	$$
	\left|2\pi \tau_\ell + (\xi_\ell\cdot p_0) 2\pi - q_0\sqrt{\delta^2 - |\alpha|^2}\right| < \frac{1}{2} \inf_{j \in \mathbb{N}_0} \frac{m_j\, j!}{\varepsilon^j (1 + \|\xi_\ell\|)^j}.
	$$

	Notice that
	\begin{align*}
		|e^{i((\xi_\ell\cdot p_0) 2\pi-q_0\sqrt{\delta^2-|\alpha|^2})}-1|^2\ \ =\ \ & \left|e^{i(2\pi \tau_\ell+(\xi_\ell\cdot p_0)2\pi-q_0\sqrt{\delta^2-|\alpha|^2})}-1\right|^2\\
		=\ \ & 2(1-\cos(2\pi \tau_\ell+(\xi_\ell\cdot p_0) 2\pi-q_0\sqrt{\delta^2-|\alpha|^2}))\\
		=\ \ & 2(2\pi \tau_\ell+(\xi_\ell\cdot p_0) 2\pi-q_0\sqrt{\delta^2-|\alpha|^2})\sin(\xi_\ell)\\
		\leq\ \ & \inf_{j\in\mathbb{N}_0}\dfrac{m_j\,j!}{\varepsilon^j(1+\|\xi_\ell\|)^j}\\
		\leq\ \ & \left[\inf_{j\in\mathbb{N}_0}\dfrac{m_j\,j!}{(\varepsilon/H)^j(1+\|\xi_\ell\|)^j}\right]^2,
	\end{align*}
	for some $\xi_\ell\in\mathbb{R}$. Then,
	\begin{equation*}
		|e^{i((\xi_\ell\cdot p_0) 2\pi-q_0\sqrt{\delta^2-|\alpha|^2})}-1|< \inf_{j\in\mathbb{N}_0}\left(\dfrac{m_j\,j!}{(\varepsilon/H)^j(1+\|\xi_\ell\|)^j}\right).
	\end{equation*}
	Therefore, (DC$_\mathscr{M}''$) does not hold.
	
	On the other hand, suppose that \eqref{DC_M_2prime} does not hold. Then, there exists $\varepsilon > 0$ such that for each $\ell \in \mathbb{N}$, there exists $\xi_\ell \in \mathbb{Z}^n$ with $\|\xi_\ell\| \geq \ell$ and
	$$
	|e^{i((\xi_\ell\cdot p_0) 2\pi- q_0\sqrt{\delta^2-|\alpha|^2})}-1| < \inf_{j \in \mathbb{N}_0} \frac{m_j \cdot j!}{\varepsilon^j(1+\|\xi_\ell\|)^j}.
	$$
	
	For each $\ell \in \mathbb{Z}$, 
	\begin{equation*}
		-\tau_\ell = \left\lfloor \xi_\ell\,p_0-\frac{q_0\sqrt{\delta^2-|\alpha|^2}}{2\pi} \right\rfloor \in \mathbb{Z},
	\end{equation*}
	that is, $-\tau_\ell$ is the integer part of $\xi_\ell\,p_0-\frac{1}{2\pi}q_0\sqrt{\delta^2-|\alpha|^2}$. 
	
	Notice that
	\begin{equation*} 
		0 \leq \xi_\ell\,p_0-\frac{q_0\sqrt{\delta^2-|\alpha|^2}}{2\pi}+\tau_\ell < 1.
	\end{equation*}

	We have 
	\begin{equation*}
		|e^{i(\xi_\ell\,p_02\pi-q_0\sqrt{\delta^2-|\alpha|^2})}-1|<\inf_{j\in\mathbb{N}_0}\dfrac{m_j\,j!}{\varepsilon^j(1+\|\xi_\ell\|)^j}\leq \dfrac{1}{\varepsilon(1+\|\xi_\ell\|)}\to 0,\quad \ell\to\infty, 
	\end{equation*}
	since $\|\xi_\ell\|\geq \ell\to\infty$ with $\ell\to\infty$. Then,
	\begin{equation*}
		\xi_\ell\,p_0-\dfrac{q_0\sqrt{\delta^2-|\alpha|^2}}{2\pi}+\tau_\ell\to 0\quad\text{or}\quad \xi_\ell\,p_0-\dfrac{q_0\sqrt{\delta^2-|\alpha|^2}}{2\pi}+\tau_\ell\to 1,\quad \ell\to\infty.
	\end{equation*} 
	
	Then, taking $\ell \in \mathbb{N}$ sufficiently large, we have
	\begin{align*}
		&\left|e^{i(2\pi \tau_\ell+\xi_\ell\,p_02\pi-q_0\sqrt{\delta^2-|\alpha|^2})}-1\right|^2 \\
		&= 2\left(1-\cos(2\pi \tau_\ell+\xi_\ell\,p_02\pi-q_0\sqrt{\delta^2-|\alpha|^2})\right) \\
		&\geq \frac{1}{2}\left|2\pi \tau_\ell+\xi_\ell\,p_02\pi-q_0\sqrt{\delta^2-|\alpha|^2}\right|^2,
	\end{align*}
	using the fact that $1-\cos(\theta)\geq \theta^2/4$ for all $\theta$ in a neighborhood of the origin. This follows from the Taylor expansion of $1-\cos(\theta)$. Since
	\begin{equation*} 
		\left|e^{i(2\pi \tau_\ell+\xi_\ell\,p_02\pi-q_0\sqrt{\delta^2-|\alpha|^2})}-1\right| < \inf_{j\in\mathbb{N}_0}\frac{m_j\,j!}{\varepsilon^j(1+\|\xi_\ell\|)^j},\quad \ell\in\mathbb{N},
	\end{equation*}
	we have
	\begin{equation*} 
		\left|2\pi \tau_\ell+\xi_\ell\,p_02\pi-q_0\sqrt{\delta^2-|\alpha|^2}\right| < 2\inf_{j\in\mathbb{N}_0}\frac{m_j\,j!}{\varepsilon^j(1+\|\xi_\ell\|)^j},\quad \ell\in\mathbb{N}.
	\end{equation*}
	
	Therefore, \eqref{DC_Mprime} does not hold. We conclude that \eqref{DC_Mprime} and \eqref{DC_M_2prime} are equivalent.
\end{proof}

\appendix \section{Technical results}

\begin{lemma}\label{lema2.25}
	Let $\mathscr{M}=\{m_j\}_{j\in\mathbb{N}_0}$ a weight sequence and $H>0$ from property (W3). Then, for all $\rho>0$, we have
	$$\left(\sup_{j\in\mathbb{N}_0}\dfrac{\rho^j}{m_j\,j!}\right)^2\leq \sup_{j\in\mathbb{N}_0}\dfrac{\rho^j H^j}{m_j\,j!}.$$
\end{lemma}
\begin{proof}
	\cite[Proposition 3.6]{kom}.
\end{proof}

\begin{remark}
	Equivalently, we have
	$$\left(\inf_{j\in\mathbb{N}_0}\dfrac{m_j\,j!}{\rho^j}\right)^2 \geq \inf_{j\in\mathbb{N}_0}\dfrac{m_j\,j!}{\rho^j H^j}.$$
	
	In particular, given $\delta>0$, taking $\rho=\delta/H$, we obtain
	$$\left(\inf_{j\in\mathbb{N}_0}\dfrac{m_j\,j!}{(\delta/H)^j}\right)^2 \geq \inf_{j\in\mathbb{N}_0}\dfrac{m_j\,j!}{\delta^j}.$$
\end{remark}

\begin{lemma}[Faà di Bruno Formula]\label{faa}
	Given $f\in C^\infty(\mathbb{R})$ and $k\in\mathbb{N}$, we have that
	$$\dfrac{d^k}{dt^k}e^{f(t)}=e^{f(t)}\displaystyle\sum_{\gamma\in\Delta(k)}\dfrac{k!}{\gamma!}\prod_{\ell=1}^{k}\left(\dfrac{1}{\ell!}\dfrac{d^\ell}{dt^\ell}f(t)\right)^{\gamma_\ell},$$
	with $\Delta(k) = \left\{\gamma=(\gamma_1,\dots,\gamma_k)\in\mathbb{N}_0^k\,:\,\sum_{\ell=1}^{k}\ell\gamma_\ell=k\right\}$.
\end{lemma}

\begin{lemma}\label{lema_sum}
	Given $k\in\mathbb{N}_0$ and $R>0$, we have that
	\begin{equation*} 
		\sum_{\gamma\in\Delta(k)}\dfrac{|\gamma|!}{\gamma!}R^{|\gamma|}=R(1+R)^{k-1},
	\end{equation*}
	with $\Delta(k)=\left\{\gamma=(\gamma_1,\dots,\gamma_k)\in\mathbb{N}_0^k\,:\, \sum_{\ell=1}^{k}\ell\gamma_\ell=k\right\}$.
\end{lemma}
\begin{proof}
	\cite[Lemma 1.4.1]{krantz}.
\end{proof}

\begin{lemma}\label{lema_prod_m}
	Let $\mathscr{M}$ a weight sequence and $\gamma\in\Delta(k)$. Denote $|\gamma|=\gamma_1+\cdots+\gamma_k$. Then,
	$$m_{|\gamma|}m_1^{\gamma_1}\cdots m_k^{\gamma_k} \leq m_k.$$
\end{lemma}
\begin{proof}
	\cite[Proposition 4.4]{bier}.
\end{proof}

\noindent \textbf{Acknowledgments.} 
 The first and second authors were supported in part by CNPq - Brasil (grants 316850/2021-7 and 423458/2021-3). This study was financed in part by the Coordenação de Aperfeiçoamento de Pessoal de Nível Superior - Brasil (CAPES) - Finance Code 001.
 
\bibliographystyle{plain}
\bibliography{references} 

\end{document}